\documentclass[11pt]{amsart}
\usepackage{amsmath}
\usepackage{amssymb}
\usepackage{amsfonts}
\usepackage{amsthm}
\usepackage{enumerate}
\usepackage[mathscr]{eucal}
\usepackage{verbatim}
\usepackage{amsthm}
\usepackage{amscd}

\newtheorem{theorem}{Theorem}[section]

\newtheorem{corollary}[theorem]{Corollary}
\newtheorem{lemma}[theorem]{Lemma}

\theoremstyle{definition}

\newtheorem{remark}[theorem]{Remark}

\numberwithin{equation}{section}

\newtheorem{innercustomgeneric}{\customgenericname}
\providecommand{\customgenericname}{}
\newcommand{\newcustomtheorem}[2]{%
  \newenvironment{#1}[1]
  {%
   \renewcommand\customgenericname{#2}%
   \renewcommand\theinnercustomgeneric{##1}%
   \innercustomgeneric
  }
  {\endinnercustomgeneric}
}

\newcustomtheorem{customthm}{Theorem}

\newcommand{\newcustomlemma}[2]{%
  \newenvironment{#1}[1]
  {%
   \renewcommand\customgenericname{#2}%
   \renewcommand\theinnercustomgeneric{##1}%
   \innercustomgeneric
  }
  {\endinnercustomgeneric}
}

\newcustomlemma{customlemma}{Lemma}

\newcommand\relphantom[1]{\mathrel{\phantom{#1}}}

\begin{document}

\address{School of Mathematics \\
           Korea Institute for Advanced Study, Seoul\\
             Republic of Korea}
   \email{qkrqowns@kias.re.kr}
\author{Bae Jun Park}

\title[boundedness of $\Psi$DO's of type (1,1) on Triebel-Lizorkin and Besov spaces ]{Sharp estimates for pseudo-differential operators of type (1,1) on Triebel-Lizorkin and Besov spaces}

\begin{abstract} 
Pseudo-differential operators of type $(1,1)$ and order $m$ are continuous from $F_p^{s+m,q}$ to $F_p^{s,q}$ if $s>d/\min{(1,p,q)}-d$ for $0<p<\infty$, and from $B_p^{s+m,q}$ to $B_{p}^{s,q}$ if $s>d/\min{(1,p)}-d$ for $0<p\leq\infty$. 
In this work we extend the $F$-boundedness result to $p=\infty$. Additionally, we prove that the operators map $F_{\infty}^{m,1}$ into $bmo$ when $s=0$, and consider H\"ormander's twisted diagonal condition for arbitrary $s\in\mathbb{R}$.
We also prove that the restrictions on $s$ are necessary conditions for the boundedness to hold.

\end{abstract}

\subjclass[2010]{Primary 42B35; Secondary 42B37}

\keywords{pseudo-differential operator, Triebel-Lizorkin spaces}

\maketitle
\section{\textbf{Introduction}}\label{intro}

Let $S(\mathbb{R}^d)$ denote the Schwartz space and $S'(\mathbb{R}^d)$ the space of tempered distributions. 
A symbol $a$ in H\"ormander's class $\mathcal{S}^m_{{\rho},\delta}$ is a smooth function defined on $\mathbb{R}^d\times\mathbb{R}^d$, satisfying that for all multi-indices ${\alpha}$ and $\beta$ there exists a constant $c_{{\alpha},\beta}$ such that 
\begin{equation*}
|\partial_{\xi}^{{\alpha}}\partial_{x}^{\beta}a(x,\xi)|\leq c_{{\alpha},\beta}(1+|\xi|)^{m-{\rho}|{\alpha}|+\delta|\beta|} ~ \text{for}~ (x,\xi)\in\mathbb{R}^d\times \mathbb{R}^d,
\end{equation*} 
and the corresponding pseudo-differential operator $T_{[a]}$ is given by 
\begin{equation*}
T_{[a]}f(x)=\int_{\mathbb{R}^d}{a(x,\xi)\widehat{f}(\xi)e^{2\pi i \langle x,\xi \rangle}}d\xi, \quad f\in {S}(\mathbb{R}^d).
\end{equation*}
Denote by  $Op\mathcal{S}_{{\rho},\delta}^m$ the class of pseudo-differential operators with symbols in $\mathcal{S}_{{\rho},\delta}^{m}$.

It is well known that for $0\leq \delta\leq{\rho}\leq 1$ the operator $T_{[a]}\in Op\mathcal{S}_{{\rho},\delta}^{m}$ maps $S$ continuously into itself. Furthermore, unless $\delta={\rho}=1$ the adjoint operator of $T_{[a]}\in Op\mathcal{S}_{{\rho},\delta}^{m}$ belongs to the same type of pseudo-differential operators (see \cite[Appendix]{Park} and \cite[p.94]{Ho1}) and  thus $T_{[a]}$ extends via duality to a mapping from $S'$ into itself.

However when ${\rho}=\delta=1$ it has many different situations. Bourdaud \cite{Bou1, Bou2} has shown that they are not closed under taking adjoints.
 In this case it was proved by Ching \cite{Chi} that not all operators of order $0$ are $L^2$ continuous. Afterwards Stein first proved that all operators in $Op\mathcal{S}_{1,1}^{0}$ are bounded on $H^{{s}}(=L_{{s}}^2)$ for $s>0$ in his unpublished work and Meyer \cite{Me} improved this result by proving the continuity of $Op\mathcal{S}_{1,1}^m$ from ${L}^p_{{s}+m}$ to ${L}^p_{{s}}$ with ${s}>0$ for $1<p<\infty$. 
Bourdaud \cite{Bou2} gave a simplified proof of the result of Meyer and also showed, by duality and interpolation, that if  $T_{[a]}, (T_{[a]})^*\in Op\mathcal{S}_{1,1}^{0}$, then $T_{[a]}$ is bounded on $L_s^p$ for all $s\in\mathbb{R}$ and $1<p<\infty$, in particular on $L^2$.
The operators in $Op\mathcal{S}_{1,1}^{0}$ are singular integral operators and by using $T1$ theorem by David and Journ\'e \cite{Da_Jo}, $T_{[a]}$ is bounded in $L^2$ if and only if $(T_{[a]})^*1\in BMO$, which is actually a weaker condition than Bourdaud's $(T_{[a]})^*\in Op\mathcal{S}_{1,1}^{0}$.
  H\"ormander \cite{Ho} introduced the twisted diagonal $\{(\xi,\eta):\xi+\eta=0\}$ in order to give the boundedness from $H^{s+m}(=L_{s+m}^2)$ to $H^s$ for all $s,m\in\mathbb{R}$. That is,
  $T_{[a]}\in Op\mathcal{S}_{1,1}^{m}$ maps $H^{s+m}$ to $H^s$ with arbitrary $s\in\mathbb{R}$ if $a$ satisfies the twisted diagonal condition
\begin{equation}\label{further}
\widehat{a}(\eta,\xi)=0 \quad \text{where}\quad C\big(|\eta+\xi|+1 \big)\leq |\xi| 
\end{equation} for some $C>1$.
 More detailed references may be found in \cite{Ho2, Ho3, Jo2}.

The continuity of the operators of type $(1,1)$ in Besov space and Triebel-Lizorkin space has been developed by Runst \cite{Ru}, Torres \cite{To}, and Johnsen \cite{Jo}.
Let 
\begin{equation*}
{\tau}_{p,q}:=\frac{d}{\min{(1,p,q)}}-d\qquad\text{and}\qquad  \tau_{p}:=\frac{d}{\min{(1,p)}}-d.
\end{equation*}

\begin{customthm}{A}\label{theoremA}
Let $m\in\mathbb{R}$ and $0<p,q\leq \infty$. Suppose $a\in\mathcal{S}_{1,1}^{m}$. Then
\begin{enumerate}
\item $T_{[a]}$ maps $F_p^{s+m,q}$ to $F_p^{s,q}$ if $p<\infty$ and  $s>\tau_{p,q}$.
\item $T_{[a]}$ maps $B_p^{s+m,q}$ to $B_p^{s,q}$ if $s>\tau_{p}$.
\end{enumerate}
\end{customthm}

Since the adjoint operators of $Op\mathcal{S}_{1,1}^{m}$ do not belong to the same class we may not generally use duality argument to define the operators on $S'$.
For $0<p,q<\infty$ and $s\in\mathbb{R}$, $S$ is dense in $F_p^{s,q}$, $B_p^{s,q}$, and thus $T_{[a]}$ is well-defined on the spaces by density arguments. 
For $0<p<\infty$, $q=\infty$ and $s>\tau_{p,q}$, we choose $\epsilon>0$ so that $s-\epsilon>\tau_{p,q}$. Then the embedding $F_p^{s+m,\infty}\hookrightarrow F_p^{s+m-\epsilon,p}$ and Theorem A with $0<p=q<\infty$ imply that $T_{[a]}$ is well-defined on $F_p^{s+m,\infty}$ and, actually, $T_{[a]}: F_{p}^{s+m,\infty}\to S'$. Accordingly, the embedding $B_p^{s+m,\infty}\hookrightarrow F_p^{s+m-\epsilon,p}$ ensures that $T_{[a]}:B_p^{s+m,\infty}\to S'$ for $0<p<\infty$ and $s>\tau_p$.
The problem happens when $p=\infty$. Theorem A (2) was first proved in \cite{Ru}, but in the paper a flaw has been detected to define $T_{[a]}\in Op\mathcal{S}_{1,1}^{m}$ on $S'$. See the remark in \cite[p. 16]{Ru} (the flaw is just about the definition of $T_{[a]}\in Op\mathcal{S}_{1,1}^{m}$, and his arguments are still valid with the adaption of (\ref{defdef}) below).
Later, Johnsen \cite{Jo, Jo1} pointed out this issue, and gave a rigorous definition of $T_{[a]}$ on $S'$ by using a limiting argument. Let $a_{j,k}(x,\xi):=\phi_j\ast a(\cdot ,\xi)(x)\widehat{\phi_k}(\xi)$ be defined as in (\ref{defpse}) below. Then we define $T_{[a]}$ by 
\begin{equation}\label{defdef}
T_{[a]}f:=\lim_{N\to\infty}{\sum_{k=0}^{N}{\sum_{j=0}^{N}{T_{[a_{j,k}]}f}}}, \quad f\in S'
\end{equation}  whenever  the limit converges in $S'$. 
Note that the case $p=\infty$ for Besov spaces was studied in Theorem A, but it is unknown for $F_{\infty}$-spaces.

Another approach was given by Torres \cite{To}, who applied atoms and molecules for $F_p^{s,q}$, introduced by Frazier and Jawerth \cite{Fr_Ja}.
 Every $f\in F_{p}^{s,q}$ can be written as $f=\sum_{Q}{s_Q A_Q}$ where $\{s_Q\}_{Q}$ is a sequence of complex numbers in $f_{p}^{s+m,q}$ and $A_Q$'s are atoms for $F_{p}^{s+m,q}$ (see \cite{Fr_Ja} for more detail). Then Torres defined 
 \begin{equation*}
 \widetilde{T_{[a]}}f:=\sum_{Q}{s_Q T_{[a]}A_Q}
 \end{equation*} and proved that $\widetilde{T_{[a]}}$ maps $F_{p}^{s+m,q}$ to $F_p^{s,q}$ by showing
$T_{[a]}$ maps atoms for $F_{p}^{s+m,q}$ to molecules for $F_{p}^{s,q}$. His argument also works for $p=\infty$.
Note that $\widetilde{T_{[a]}}$ agrees with $T_{[a]}$ on $F_p^{s,q}$ for $0<p<\infty$ by the density argument for $0<p<\infty$. However, we should be careful to say that this implies the boundedness of the operator $T_{[a]}$ when $p=\infty$ because $T_{[a]}$ is not continuous on $S'$ as we mentioned above.\\

In this paper we extend $F$-boundedness results in Theorem A to $p=\infty$ with the adaption of (\ref{defdef}). 
That is, $T_{[a]}\in Op\mathcal{S}_{1,1}^{m}$ maps $F_{\infty}^{s+m,q}$ into $F_{\infty}^{s,q}$ for $s>\tau_{q}( =\tau_{\infty,q} )$ and $m\in\mathbb{R}$.
We also prove that the condition $s>\tau_q$ can be dropped under the H\"ormander's twisted diagonal condition and furthermore, when $s=0$, $T_{[a]}\in Op\mathcal{S}_{1,1}^{m}$ maps $F_{\infty}^{m,q}$ into $bmo(=F_{\infty}^{0,2})$.\\

Another main part of the paper is about negative results. We prove that the assumptions $s>\tau_{p,q}$ and $s>\tau_p$ in Theorem A are necessary conditions for the boundedness to hold. Moreover, we also discuss the sharpness of our new boundedness result for $p=\infty$, which is an ``almost sharpness" of $F_{\infty}$-boundedness in the sense that $s>\tau_q$ is sharp when considering only a grid of dyadic cubes of a fixed side length in the definition of the $F_{\infty}$-norm.

Some negative results for $p,q>1$ may be found in \cite{Bou2, Chi} and \cite{Jo1}, and the same technique can be applied to show the optimality of $s>\tau_{p,q}=\tau_p=0$ for $p,q>1$ (see Remark \ref{remark2.8} below for more details).
However, this method is not valid for $\min{(p,q)}\leq 1$ and particular emphasis is given to the case $\min{(p,q)}\leq1$.
\\

This paper is organized as follows. The main results are stated in Section \ref{statementsection}. We give some key lemmas for the proof of our results in Section \ref{maximal}. Then we prove  our results in Section \ref{positiveresult} and Section \ref{negativeresult}.

  \section{Statements of main results}\label{statementsection}
  
\subsection*{Notations}  
We use standard notation. 
Let $\mathbb{N}$ be the collection of all natural numbers and $\mathbb{N}_0:=\mathbb{N}\cup\{0\}$. 
$\mathbb{Z}$ denotes the collection of all integers.
For $f\in S$ the Fourier transform is defined by the formula
$\widehat{f}(\xi)=\int_{\mathbb{R}^d}{f(x) e^{-2\pi i\langle x,\xi \rangle}}dx$ and denote by $f^{\vee}$ the inverse Fourier transform of $f$.  We also extend these transforms to $S'$.
 Let  $\mathcal{D}$ denote the set of all dyadic cubes in $\mathbb{R}^d$, and 
 for each $k\in\mathbb{Z}$, $\mathcal{D}_{k}$ stands for the subset of $\mathcal{D}$ consisting of the cubes with side length $2^{-k}$.
 For each $Q\in\mathcal{D}$ we denote the side length of $Q$ by $l(Q)$ and the characteristic function of $Q$ by $\chi_Q$, and for $r>0$ we denote by $rQ$ the cube concentric with $Q$ having the side length $rl(Q)$.  
The symbol $X\lesssim Y$ means that there exists a positive constant $C$, possibly different at each occurrence, such that $X\leq CY$. $X\approx Y$ means $C^{-1}Y\leq X\leq CY$ for a positive unspecified constant $C$.

\subsection*{Function spaces}
 We  recall the definitions of Besov sapces and Triebel-Lizorkin spaces from \cite{Fr_Ja} and \cite{Tr}.
Let $\phi$ be a smooth function so that $\widehat{\phi}$ is supported in $\{\xi:2^{-1}\leq |\xi|\leq 2\}$ and $\sum_{k\in\mathbb{Z}}{\widehat{\phi_k}(\xi)}=1$ for $\xi\not=0$ where $\phi_k:=2^{kd}\phi(2^k\cdot)$. 
Let $\widehat{\Phi}:=1-\sum_{k=1}^{\infty}\widehat{\phi_k}$. Then we define convolution operators $\Pi_0$ and $\Pi_k$ by
$\Pi_0f:=\Phi\ast f$ and $\Pi_kf:=\phi_k\ast f$ for $k\geq 1$.
For $0<p,q\leq \infty$ and $s\in \mathbb{R}$ the (inhomogeneous) Besov spaces ${B}_p^{s,q}$ are defined as a subspace of $S'$ with  (quasi-)norms   
\begin{equation*}
\Vert f\Vert_{{B}_p^{s,q}}:=\Vert \Pi_0f\Vert_{L^p}+\big\Vert \{2^{s k}\Pi_kf\}_{k=1}^{\infty}\big\Vert_{l^q(L^p)}.
\end{equation*}
For $0<p<\infty$, $0<q\leq \infty$, and $s\in\mathbb{R}$ we define (inhomogeneous) Triebel-Lizorkin spaces ${F}_{p}^{s,q}$ to be a subspace of $S'$ with norms  
\begin{equation*}
\Vert f\Vert_{F_p^{s,q}}:=\Vert \Pi_0f\Vert_{L^p}+\big\Vert \{2^{s k}\Pi_kf\}_{k=1}^{\infty}\big\Vert_{L^p(l^q)}. 
\end{equation*}
When $p=q=\infty$ we apply 
\begin{equation*}
\Vert f\Vert_{F_{\infty}^{s,\infty}}:=\Vert f\Vert_{B_{\infty}^{s,\infty}},
\end{equation*} and for $p=\infty$ and $q<\infty$ we employ
\begin{equation*}
\Vert f\Vert_{F_{\infty}^{s,q}}:=\Vert \Pi_0f\Vert_{L^{\infty}}+\sup_{l(P)<1}\Big(\frac{1}{|P|}\int_P{\sum_{k=-\log_2{l(P)}}^{\infty}{2^{s kq}|\Pi_kf(x)|^q}}dx \Big)^{1/q}
\end{equation*} where  the supremum is taken over all dyadic cubes whose side length is less than $1$.
According to those norms,  the spaces are quasi-Banach spaces (Banach spaces if $p\geq 1, q\geq 1$).
Note that the spaces are a generalization of many standard function spaces such as $L^p$ spaces, Sobolev spaces, and Hardy spaces. We recall  $L^p = F_p^{0,2}$ for $1<p<\infty$, $h^p = F_p^{0,2}$ for $0<p<\infty$, ${L}_s^p = {F}_p^{s,2}$ for $s>0, 1<p<\infty$, and $ bmo = {F}_{\infty}^{0,2}$.

\subsection{Positive results}
\begin{theorem}\label{mmain}
Let $m\in\mathbb{R}$, $0<q< \infty$, and $\mu\in\mathbb{N}$. Suppose $a\in\mathcal{S}_{1,1}^{m}$. If $s>{\tau}_{q}$  then 
\begin{align}\label{maininequality}
&\sup_{P\in\mathcal{D}_{\mu}}\Big(\frac{1}{|P|}\int_P\sum_{k=\mu}^{\infty}{2^{skq} \big|\Pi_kT_{[a]}f(x) \big|^q}dx \Big)^{1/q}\nonumber\\
&\lesssim\sup_{0\leq k\leq \mu-1}{\Vert 2^{k(s+m)}\Pi_kf \Vert_{L^{\infty}}}
+\sup_{R\in\mathcal{D}_{\mu}}\Big(\frac{1}{|R|}\int_R\sum_{k=\mu}^{\infty}{2^{(s+m)kq} \big|\Pi_kf(x) \big|^q} dx\Big)^{1/q}.
\end{align}

Here the implicit constant of the inequality is independent of $\mu$.
 \end{theorem}
 
From Theorem \ref{mmain} we immediately deduce the following conclusion.
\begin{corollary}\label{finftyresult}
Let $m\in\mathbb{R}$ and  $0<q<\infty$. Suppose $a\in\mathcal{S}_{1,1}^m$. If $s>\tau_q$ then
\begin{equation*}\label{Finfty}
T_{[a]} : F_{\infty}^{s+m,q}\to F_{\infty}^{s,q}.
\end{equation*}
\end{corollary}

\begin{remark}\label{remark2.3}
Theorem \ref{mmain} is sharp in the sense that if $s\leq \tau_q$ then the inequality (\ref{maininequality}) does not hold. This will be stated in Theorem \ref{sharpresult2}. In Section \ref{intro} we used the expression ``almost sharpness" of $F_{\infty}$-boundedness because we do not obtain the sharpness of $F_{\infty}$-boundedness (for $0<q\leq 1$) in Corollary \ref{Finfty}. To be specific, the case $1<q<\infty$ in Corollary \ref{finftyresult} is sharp, but we do not conclude it for $0<q\leq 1$. See Remark \ref{remark2.8} below.\\
\end{remark}

For the case $s=0$ we have the following result. 
\begin{theorem}\label{bmo}
Suppose $m\in\mathbb{R}$ and $a\in\mathcal{S}_{1,1}^{m}$. Then  $T_{[a]}$ maps $F_{\infty}^{m,1}$ into $bmo(=F_{\infty}^{0,2})$.
\end{theorem}
The main value of this theorem is that $s$ can be taken to be $0 (=\tau_q)$ provided one increases the value of $q$ in the space that $T_{[a]}$ maps in.\\

Now we extend H\"ormander's results to $F_{\infty}$-spaces which give a sufficient condition for the boundedness for arbitrary $s\in\mathbb{R}$.
\begin{theorem}\label{s=0}
Suppose $0<q< \infty$, $s,m\in\mathbb{R}$, and $a\in\mathcal{S}_{1,1}^{m}$.
If (\ref{further}) holds for $C>1$, then $T_{[a]}$ maps $F_{\infty}^{s+m,q}$ into $F_{\infty}^{s,q}$.
\end{theorem}
Note that the counterparts for $p<\infty$ may be found in \cite{Jo}.

\subsection{Negative results}\label{negastate}\hfill

The following theorem proves the optimality of $s>\tau_{p,q}$ and $s>\tau_p$ in Theorem \ref{theoremA}.
\begin{theorem}\label{sharpresult1}
Let  $0<p,q\leq\infty$ and $m\in\mathbb{R}$.
\begin{enumerate}
\item Suppose $0<p<\infty$ or $p=q=\infty$. If $s\leq\tau_{p,q}$ then there exists $a\in\mathcal{S}_{1,1}^{m}$ so that $\Vert T_{[a]} \Vert_{F_{p}^{s+m,q}\to F_{p}^{s,q}}=\infty$.
\item If  $s\leq\tau_{p}$ then there exists $a\in\mathcal{S}_{1,1}^{m}$ so that $\Vert T_{[a]} \Vert_{B_{p}^{s+m,q}\to B_{p}^{s,q}}=\infty$.
\end{enumerate}
\end{theorem}

The optimality of $s>\tau_q$ in Theorem \ref{mmain} follows from the next result.
\begin{theorem}\label{sharpresult2}
Let  $0<q< \infty$, $m\in\mathbb{R}$, and $\mu\in\mathbb{N}$.   If $s\leq\tau_q$ then there exists $a\in\mathcal{S}_{1,1}^{m}$ and $f\in S'$ such that
\begin{equation*}
\sup_{0\leq k\leq \mu-1}{\big\Vert 2^{k(s+m)}\Pi_kf \big\Vert_{L^{\infty}} }+\sup_{R\in\mathcal{D}_{\mu}}\Big(\frac{1}{|R|}\int_R\sum_{k=\mu}^{\infty}{2^{(s+m)kq} \big|\Pi_kf(x) \big|^q} \Big)^{1/q}<\infty,
\end{equation*}
\begin{equation*}
\sup_{P\in\mathcal{D}_{\mu}}\Big(\frac{1}{|P|}\int_P\sum_{k=\mu}^{\infty}{2^{skq} \big|\Pi_kT_{[a]}f(x) \big|^q} \Big)^{1/q}=\infty.
\end{equation*}

\end{theorem}

\begin{remark}\label{remark2.8}
Theorem \ref{sharpresult1} and Theorem \ref{sharpresult2} for the case $1<p,q$ can be proved by using the idea of Ching \cite{Chi}, which is to construct $a\in\mathcal{S}_{1,1}^{m}$ and $f\in F_{p}^{s+m,q}$ such that 
\begin{equation}\label{pq1}
\Vert f\Vert_{F_{p}^{s+m,q}}\lesssim \Big( \sum_{k=1}^{\infty}{|v_k|^{p}}\Big)^{1/p} \quad\text{or}\quad \Vert f\Vert_{F_{p}^{s+m,q}}\lesssim \Big( \sum_{k=1}^{\infty}{|v_k|^{q}}\Big)^{1/q} 
\end{equation} and 
\begin{equation}\label{111}
\Vert T_{[a]}f\Vert_{F_{p}^{s,q}}\gtrsim \sum_{k=1}^{\infty}{2^{-sk}|v_k|}
\end{equation} where $\{v_k\}$ is a sequence of positive numbers
($\Vert f\Vert_{B_p^{s+m,q}}\lesssim \big( \sum_{k=1}^{\infty}{|v_k|^q}\big)^{1/q}$ and $\Vert T_{[a]}f\Vert_{B_p^{s,q}}\gtrsim \sum_{k=1}^{\infty}|v_k|$ for the proof of Theorem \ref{sharpresult1} (2)).
Then choose $\{v_k\}$ such that $(\ref{pq1})<\infty$, but $(\ref{111})=\infty$ for $s\leq 0$.
 Indeed, similar sharp results for $1<p<\infty, 1<q\leq \infty$ are found in \cite{Bou2} and \cite{Jo1} where the same technique is applied.  
Furthermore, the case $1<q<\infty$ in Theorem \ref{sharpresult2} can be proved in a similar way. Indeed, for a sequence of positive numbers $\{v_k\}$ to be chosen later, let
\begin{equation*}
a(x,\xi):=\sum_{k=10}^{\infty}{2^{km}\widehat{{\phi_k^*}}(\xi)e^{2\pi i\langle 2^{-k}e_1,\xi\rangle}e^{-2\pi i\langle 2^ke_1,x\rangle}}
\end{equation*} 
\begin{equation*}
f_L(x):=\sum_{k=10}^{L}{v_{k}2^{-t_k(s+m)}\mathcal{G}(x-2^{-t_k}e_1)e^{2\pi i\langle 2^{t_k}e_1,x\rangle}}
\end{equation*}
where $t_k:=10k$, $e_1:=(1,0,\dots,0)\in\mathbb{Z}^d$, ${\phi_k^*}:=\phi_{k-1}+\phi_{k}+\phi_{k+1}$, and $\mathcal{G}\in S$ satisfies $Supp(\widehat{\mathcal{G}})\subset \{|\xi|\leq 1\}$.
 Then clearly $a\in\mathcal{S}_{1,1}^{m}$, and 
 \begin{equation*}
 \Vert f_L\Vert_{F_{\infty}^{s+m,q}}\lesssim \Big(\sum_{k=10}^{L}{v_{k}^q} \Big)^{1/q}.
 \end{equation*}
 Moreover, due to the fact that $\widehat{{\phi_{t_k}^*}}(\xi)\widehat{\mathcal{G}}(\xi-2^{t_k}e_1)=\widehat{\mathcal{G}}(\xi-2^{t_k}e_1)$ one has
 \begin{equation*}
 T_{[a]}f_L(x)=\mathcal{G}(x)\sum_{k=10}^{L}{2^{-st_k}v_k}
 \end{equation*} and this yields
 \begin{equation*}
 \big\Vert T_{[a]}f_L\big\Vert_{F_{\infty}^{s,q}}\approx \sum_{k=10}^{L}{2^{-st_k}v_k}.
 \end{equation*}
 Choose a sequence $\{v_k\}$ so that $\sum_{k=10}^{\infty}{2^{-st_k}v_k}=\infty$ and $\sum_{k=10}^{\infty}{v_k^q}<\infty$. By letting $L\to\infty$ we are done.
 Actually, this proves the sharpness of Corollary \ref{finftyresult} for $1<q<\infty$.

 However, when $\min{(p,q)}\leq 1$ this method does not work and it is not easy to verify the optimality of $s>\tau_{p,q}=d/\min{(p,q)}-d$ and $s>\tau_q=d/q-d$. 
 In this paper we will consider only the case $\min{(p,q)}\leq 1$. 
\end{remark}

\section{\textbf{Maximal inequalities  and Key lemmas}}\label{maximal}

\subsection{Maximal inequalities}

A crucial tool in theory of function spaces is maximal inequalities of Fefferman-Stein \cite{Fe_St} and Peetre \cite{Pe}.

 Let $\mathcal{M}$  be the Hardy-Littlewood maximal operator, defined by
 \begin{equation*}
 \mathcal{M}f(x):=\sup_{x\in Q}{\frac{1}{|Q|}\int_Q{|f(y)|}dy}
 \end{equation*} where the supremum is taken over all cubes containing $x$, and for $0<t<\infty$ let $\mathcal{M}_tf:=\big(  \mathcal{M}(|f|^t) \big)^{1/t}$.
Then Fefferman-Stein's vector-valued maximal inequality in \cite{Fe_St} says that for $0<r<p,q<\infty$ one has
\begin{equation}\label{hlmax}
\Big\Vert  \Big(\sum_{k}{(\mathcal{M}_{r}f_k)^q}\Big)^{1/{q}} \Big\Vert_{L^p} \lesssim  \Big\Vert \Big( \sum_{k}{|f_k|^q}  \Big)^{1/{q}}  \Big\Vert_{L^p}.
\end{equation} 
Note that (\ref{hlmax}) also holds when $q=\infty$.

Now for $k\in\mathbb{Z}$ and $\sigma>0$ we define the Peetre maximal operator $\mathfrak{M}_{\sigma,2^k}$ by 
\begin{equation*}
\mathfrak{M}_{\sigma,2^k}f(x):=\sup_{y\in\mathbb{R}^d}{\frac{|f(x-y)|}{(1+2^k|y|)^{\sigma}}}.
\end{equation*}
For $r>0$ let $\mathcal{E}(r)$ be the space of tempered distributions whose Fourier transforms are supported in $\{\xi:|\xi|\leq 2r\}$. 
As shown in \cite{Pe} one has the majorization 
\begin{equation}\label{maximalbound}
\mathfrak{M}_{d/r,2^k}f(x)\lesssim \mathcal{M}_rf(x) 
\end{equation}  if  $f\in \mathcal{E}(2^k)$. Then  via (\ref{hlmax}) and (\ref{maximalbound}) the following maximal inequality holds. Suppose $0<p<\infty$, $0<q\leq \infty$. Then for $f_k\in\mathcal{E}(2^k)$
\begin{equation}\label{max}
\Big\Vert  \Big(\sum_{k}{(\mathfrak{M}_{d/r,2^k}f_k)^q}\Big)^{1/{q}} \Big\Vert_{L^p} \lesssim  \Big\Vert \Big( \sum_{k}{|f_k|^q}  \Big)^{1/{q}}  \Big\Vert_{L^p} \quad \text{for}~r<p,q.
\end{equation}

For $\epsilon\geq0$, $r>0$, and $k\in\mathbb{Z}$, we now introduce a variant of Hardy-Littlewood maximal operators $\mathcal{M}_r^{k,\epsilon}$, which is defined by
\begin{align*}
\mathcal{M}_r^{k,\epsilon}f(x)&:=\sup_{2^kl(Q)\leq 1}{\Big( \frac{1}{|Q|}\int_{Q}{|f(x)|^r}dx  \Big)^{1/r}}\\
  &\relphantom{=} +\sup_{2^{k}l(Q)>1}{\big(2^kl(Q)\big)^{-\epsilon}\Big( \frac{1}{|Q|}\int_Q{|f(x)|^r}dx  \Big)^{1/r}}. \nonumber
\end{align*}
Note that $\mathcal{M}_r^{k,0}f(x)\approx \mathcal{M}_rf(x)$. It is proved in \cite{Park1} that for $r<t$ and $f\in \mathcal{E}(A2^k)$ for some $A>0$
\begin{equation}\label{pointmaximal}
\mathfrak{M}_{d/r,2^k}f(x)\lesssim_A \mathcal{M}_t^{k,d/r-d/t}f(x)\lesssim \mathcal{M}_tf(x) \quad \text{uniformly in}~ k.
\end{equation}
Furthermore, the following maximal inequalities hold.
\begin{lemma}\label{maximal1}\cite{Park1}
Let $0<r<q<\infty$, $\epsilon>0$, and $\mu\in\mathbb{Z}$. For $k\in\mathbb{Z}$ let $f_k\in\mathcal{E}(A2^k)$  for some $A>0$.
Then one has
\begin{equation*}
\sup_{P\in\mathcal{D}_{\mu}}{\Big(  \frac{1}{|P|}\int_P{  \sum_{k=\mu}^{\infty}{  \big( \mathcal{M}_r^{k,\epsilon}f_k(x)  \big)^{q}      }    }dx  \Big)^{1/q}}  \lesssim_A \sup_{R\in\mathcal{D}_{\mu}}{\Big(  \frac{1}{|R|}\int_R{  \sum_{k=\mu}^{\infty}{   |f_k(x)|^q   }    }dx  \Big)^{1/q}}. 
\end{equation*}
Here, the implicit constant of the inequality is independent of $\mu$.
\end{lemma}

As an immediate consequence of the pointwise estimate (\ref{pointmaximal}) and Lemma \ref{maximal1} one obtains
\begin{lemma}\label{maximal2}\cite{Park1}
Let $0<r<q<\infty$ and $\mu\in\mathbb{Z}$. For $k\in\mathbb{Z}$ let  $f_k\in\mathcal{E}(A2^k)$  for some $A>0$.
Then one has
\begin{equation*}
\sup_{P\in\mathcal{D}_{\mu}}{\Big(  \frac{1}{|P|}\int_P{  \sum_{k=\mu}^{\infty}{  \big(\mathfrak{M}_{d/r,2^k}f_k(x) \big)^{q}      }    }dx  \Big)^{1/q}} \lesssim_A \sup_{R\in\mathcal{D}_{\mu}}{\Big(  \frac{1}{|R|}\int_R{  \sum_{k=\mu}^{\infty}{   |f_k(x)|^q   }    }dx  \Big)^{1/q}}.
\end{equation*}
Here, the implicit constant of the inequality is independent of $\mu$.
\end{lemma}

Note that Lemma \ref{maximal2} is sharp in the sense that if $r\geq q$ then there exists a sequence $\{f_k\}$ in $\mathcal{E}(2^k)$ for which the inequality does not hold.

As an application of Lemma \ref{maximal2}, for $\mu\in\mathbb{Z}$, $q_1<q_2<\infty$, and $\mathbf{f}:=\{f_k\}_{k\in\mathbb{Z}}$ we have
\begin{equation}\label{embinf}
\mathcal{V}_{\mu,q_2}[\mathbf{f}]\lesssim \mathcal{V}_{\mu,q_1}[\mathbf{f}],
\end{equation} provided that each $f_k$ is defined as in Lemma \ref{maximal2}, where
\begin{equation*}
\mathcal{V}_{\mu,q}[\mathbf{f}]:= \sup_{P\in\mathcal{D},l(P)\leq 2^{-\mu}}{\Big(  \frac{1}{|P|}\int_P{  \sum_{k=-\log_2{l(P)}}^{\infty}{   |f_k(x)|^q   }    }dx  \Big)^{1/q}}. \end{equation*}
Indeed,  for fixed $Q\in\mathcal{D}_{k}$ and   $\sigma>0$ 
\begin{equation}\label{qwqw}
\Vert \mathfrak{M}_{\sigma,2^k}f_k\Vert_{L^{\infty}(Q)}\lesssim_{\sigma} \inf_{y\in Q}{\mathfrak{M}_{\sigma,2^k}f_k(y)}
\end{equation} and then for $k\geq \mu$, $P\in\mathcal{D}_{\mu}$, and $\sigma>d/q_1$
\begin{align*}
 \Vert f_k\Vert_{L^{\infty}(P)}&\leq \sup_{Q\in\mathcal{D}_k, Q\subset P}{\big\Vert \mathfrak{M}_{\sigma,2^k}f_k\big\Vert_{L^{\infty}(Q)}}\\ 
 &\lesssim_{\sigma} \sup_{Q\in\mathcal{D}_k, Q\subset P}{\Big( \frac{1}{|Q|}\int_Q{\big( \mathfrak{M}_{\sigma,2^k}f_k(y)\big)^{q_1}}dy\Big)^{1/q_1}}    \lesssim \mathcal{V}_{\mu,q_1}[\mathbf{f}]
\end{align*}  where we used  Lemma \ref{maximal2} in the last inequality.
By applying $|f_k(x)|^{q_2}\lesssim \big(\mathcal{V}_{\mu,q_1}[\mathbf{f}]\big)^{q_2-q_1}|f_k(x)|^{q_1}$ for $x\in P$, one can prove (\ref{embinf}).
Furthermore, by using (\ref{qwqw}) and Lemma \ref{maximal2}  for $0<q<\infty$ and $\mu\in\mathbb{Z}$
\begin{equation*}
\sup_{k\geq \mu}\Vert f_k\Vert_{L^{\infty}}\lesssim \mathcal{V}_{\mu,q}[\mathbf{f}].
\end{equation*}
 Then together with (\ref{embinf}), this  implies 
 \begin{equation}\label{bmoembedding}
 F_{\infty}^{s,q_1}\hookrightarrow F_{\infty}^{s,q_2}, \quad 0<q_1<q_2\leq\infty.\\ \nonumber
 \end{equation}

\subsection{Key Lemmas}

One of the main tools in the proof of Theorem A is the following well-known lemma. 

\begin{customlemma}{B}\label{originalmarshall}
Let $0<p< \infty$, $0<q\leq\infty$ and $s\in\mathbb{R}$. Suppose that  $\{f_k\}_{k\in\mathbb{N}_0}$ is a family of tempered distributions. Let $A>1$.
\begin{enumerate}
\item If  $Supp(\widehat{f_0})\subset \{\xi:|\xi|\leq A\}$ and  $Supp(\widehat{f_k})\subset \{\xi:2^k/A\leq |\xi|\leq A2^k\}$ for $k\geq 1$, then
\begin{equation*}
\Big\Vert \sum_{k=0}^{\infty}{f_k} \Big\Vert_{{F}_{p}^{s,q}}\lesssim_A \big\Vert \{2^{sk}f_k\}_{k=0}^{\infty}\big\Vert_{L^p(l^q)}.
\end{equation*}

\item If  $Supp(\widehat{f_k})\subset \{\xi: |\xi|\leq A2^k\}$ for $k\in\mathbb{N}_0$ and $s>\tau_{p,q}$, then
\begin{equation*}
\Big\Vert \sum_{k=0}^{\infty}{f_k} \Big\Vert_{{F}_{p}^{s,q}}\lesssim_A \big\Vert \{2^{sk}f_k\}_{k=0}^{\infty}\big\Vert_{L^p(l^q)}.
\end{equation*}

\end{enumerate}

\end{customlemma}
The first assertion follows immediately from the Nikol'skii representation (see \cite[2.5.2]{Tr} for details). The second one was proved in \cite{Ya}. See also  \cite{Ma3}, \cite{Ma1}, and \cite{Me}.

The following two lemmas are the counterpart of Lemma \ref{originalmarshall} for $F_{\infty}^{s,q}$.

\begin{lemma}\label{marshall1}
Let $0<q< \infty$, ${s} \in \mathbb{R}$, $A>1$. Let $\{f_k\}_{k\in\mathbb{N}_0}$  be a sequence of tempered distributions satisfying $Supp(\widehat{f_0})\subset \{\xi: |\xi|\leq A\}$, $Supp(\widehat{f_k})\subset \{\xi:2^{k}/A\leq |\xi|\leq 2^{k}A\}$ for $k\geq 1$.  Fix $\mu\in\mathbb{N}$ and $P\in\mathcal{D}_{\mu}$. Then there exists $h\in\mathbb{N}$ such that
\begin{equation*}
\Big( \frac{1}{|P|}\int_{P}{\sum_{n=\mu}^{\infty}{2^{snq}\Big|\Pi_n\Big( \sum_{k=0}^{\infty}f_k\Big)(x) \Big|^q}}dx\Big)^{1/q} \lesssim \sup_{R\in\mathcal{D}_{\mu}}{\Big(\frac{1}{|R|}\int_R{\sum_{k=\max{(0,\mu-h)}}^{\infty}{2^{skq}|f_k(x)|^q}}dx\Big)^{1/q}}
\end{equation*} 
where the implicit constant in the inequality is independent of $\mu$ and $P$.
\end{lemma}

\begin{lemma}\label{marshall2}

Let $0<q< \infty$, $s>{\tau}_q$, and $A>1$. Let $\{f_k\}_{k\in\mathbb{N}_0}$  be a sequence of tempered distributions satisfying $Supp(\widehat{f_k})\subset \{\xi: |\xi|\leq A2^k\}$ for $k\in\mathbb{N}_0$.  Fix $\mu\in\mathbb{N}$ and $P\in\mathcal{D}_{\mu}$.
Then there exists $h\in\mathbb{N}$ such that
\begin{equation*}
\Big( \frac{1}{|P|}\int_{P}{\sum_{n=\mu}^{\infty}{2^{snq}\Big|\Pi_n\Big( \sum_{k=0}^{\infty}f_k\Big)(x) \Big|^q}}dx\Big)^{1/q} \lesssim \sup_{R\in\mathcal{D}_{\mu}}{\Big(\frac{1}{|R|}\int_R{\sum_{k=\max{(0,\mu-h)}}^{\infty}{2^{skq}|f_k(x)|^q}}dx\Big)^{1/q}}
\end{equation*}
where the implicit constant in the inequality is independent of $\mu$ and $P$.
\end{lemma}

\begin{remark}\label{remark3.5}
It has been observed in \cite[Lemma 13]{Ma2} that weaker versions hold, namely that the right hand side of the inequalities in Lemma \ref{marshall1} and \ref{marshall2} are supremums over arbitrary dyadic cubes, not over $R\in\mathcal{D}_{\mu}$. Thus, the above lemmas improve the ones in \cite[Lemma 13]{Ma2}.
Moreover, the proof of such earlier versions are based on the vector-valued Fourier multiplier theorem.
We will provide a different and elementary proof by just using Lemma \ref{maximal1} and \ref{maximal2}.
\end{remark}

We prove only Lemma \ref{marshall2} because  Lemma \ref{marshall1} can be proved in a similar and simpler way.
Our proof is based on the idea in \cite{Ya}.

\begin{proof}[Proof of Lemma \ref{marshall2}]
We fix $\mu\in\mathbb{N}$ and  $P\in\mathcal{D}_{\mu}$, and  regard $f_k\equiv 0$ when $k<0$ in the proof.
Let $h$ be the smallest integer  greater or equal to $ \log_2A$.
The supports of $\widehat{\phi_n}$ and $\widehat{f_k}$ ensure that the left hand side of the desired inequality is less than
\begin{equation}\label{adad}
{\Big( \frac{1}{|P|}\int_P{ \sum_{n=\mu}^{\infty}{2^{snq}\Big(   \sum_{k=n-h}^{\infty}{|\Pi_n f_k(x)|}     \Big)^q}    }dx   \Big)^{1/q}}.
\end{equation} 
We choose $0<r<\min{(1,q)}$ such that $s>d/r-d>\tau_q$, and then pick $0<\epsilon<d/r$ so that $s>d/r-d+\epsilon r$. Let $\sigma>d/r$.
For each $k\geq n-h$ we see that
\begin{align*}
\big|\Pi_nf_k(x)\big|&\leq\int_{\mathbb{R}^d}{\big( 1+2^n|y| \big)^{{\sigma}}|\phi_n(y)|\frac{|f_k(x-y)|}{\big( 1+2^n|y| \big)^{{\sigma}}}}dy\\
  &\leq\big\Vert  (1+2^n|\cdot|)^{{\sigma}}\phi_n  \big\Vert_{L^{\infty}}\Big( \sup_{y}{\frac{|f_k(x-y)|}{(1+2^n|y|)^{{\sigma}}}} \Big)^{1-r}\int_{\mathbb{R}^d}{\frac{|f_k(x-y)|^r}{(1+2^n|y|)^{{\sigma} r}}}dy\\
  &\lesssim 2^{nd}\Big( \sup_{y}{\frac{|f_k(x-y)|}{(1+2^n|y|)^{{\sigma}}}} \Big)^{1-r}\int_{\mathbb{R}^d}{\frac{|f_k(x-y)|^r}{(1+2^n|y|)^{{\sigma} r}}}dy.
\end{align*}
Let $t$ satisfy $\epsilon=d/r-d/t>0$. Then by (\ref{pointmaximal})
\begin{equation*}
 \sup_{y}{\frac{|f_k(x-y)|}{(1+2^n|y|)^{{\sigma}}}} \lesssim   2^{(k-n)d/r} \mathfrak{M}_{d/r,2^k}f_k(x)\lesssim 2^{(k-n)d/r}  \mathcal{M}_{t}^{k,\epsilon}f_k(x).
\end{equation*}
Moreover, for $j,n\in\mathbb{Z}$ let $E_j^n(x):=\{y: 2^{j-1}<2^n|x-y|\leq 2^j\}$. Then
\begin{align*}
&\int_{\mathbb{R}^d}{\frac{|f_k(x-y)|^r}{(1+2^n|y|)^{{\sigma} r}}}dy=\int_{\mathbb{R}^d}{\frac{|f_k(y)|^r}{(1+2^n|x-y|)^{{\sigma} r}}}dy\\
&\lesssim\sum_{j=-\infty}^{n-k}{\frac{1}{(1+2^j)^{{\sigma} r}} \int_{E_j^n(x)}|f_k(y)|^rdy    }+\sum_{j=n-k+1}^{\infty}{\frac{1}{(1+2^j)^{{\sigma} r}} \int_{E_j^n(x)}|f_k(y)|^rdy    }\\
&=\sum_{j=-\infty}^{n-k}{\frac{2^{(j-n)d}}{(1+2^j)^{{\sigma} r}} \frac{1}{2^{(j-n)d}}\int_{E_j^n(x)}|f_k(y)|^rdy    }+\sum_{j=n-k+1}^{\infty}{\frac{2^{(j-n)(d+\epsilon r)}}{(1+2^j)^{{\sigma} r}} \frac{1}{2^{(j-n)(d+\epsilon r)}}\int_{E_j^n(x)}|f_k(y)|^rdy    }\\
&\lesssim 2^{-nd}\Big( \sup_{\substack{V\ni x,\\ l(V)\leq 2^{-k}}}{\frac{1}{|V|}\int_V{|f_k(y)|^r}dy}+2^{-n\epsilon r}\sup_{\substack{V\ni x,\\ l(V)>2^{-k}}}{\frac{1}{|V|^{1+\epsilon r/d}}\int_V{|f_k(y)|^r}dy}  \Big)\\
&\lesssim 2^{-nd}\Big[ \Big(\sup_{\substack{V\ni x,\\ l(V)\leq 2^{-k}}}{\frac{1}{|V|}\int_V{|f_k(y)|^t}dy}\Big)^{r/t}+2^{-n\epsilon r}\Big(\sup_{\substack{V\ni x,\\ l(V)>2^{-k}}}{\frac{1}{|V|^{1+\epsilon t/d}}\int_V{|f_k(y)|^t}dy}  \Big)^{r/t}\Big]\\
&\lesssim 2^{-nd}2^{(k-n)\epsilon r}\big( \mathcal{M}_{t}^{k,\epsilon}f_k(x)\big)^{r}
\end{align*} for $k\geq n-h$, where the third inequality follows by H\"older's inequality.
By putting these together one obtains
\begin{equation}\label{plpl}
\big| \Pi_n f_k(x) \big|\lesssim 2^{(k-n)d(1/r-1)}2^{(k-n)\epsilon r}\mathcal{M}_t^{k,\epsilon}f_k(x).
\end{equation}
Now choose $\delta>0$ sufficiently small so that $0<\delta<s-(d/r-d)-r\epsilon$.
Then it follows, from H\"older's inequality if $q>1$ or embedding $l^q\hookrightarrow l^1$ if $q\leq 1$, that
\begin{equation*}
\sum_{n=\mu}^{\infty}{2^{snq}\Big( \sum_{k=n-h}^{\infty}{\big| \Pi_n f_k(x)  \big|}   \Big)^q}\lesssim \sum_{n=\mu}^{\infty}{2^{snq} 2^{-\delta nq}\sum_{k=n-h}^{\infty}{2^{\delta kq }\big| \Pi_n f_k(x)  \big|^q}   }
\end{equation*} and by (\ref{plpl}) and the choice of $\delta$, the last expression is bounded by a constant times
\begin{equation*}
\sum_{k=\mu-h}^{\infty}{2^{kq(d/r-d)}2^{\epsilon rkq}2^{\delta kq}\big(  \mathcal{M}_t^{k,\epsilon}f_k(x)  \big)^q \sum_{n=\mu}^{k+h}{2^{nq(s-(d/r-d)-\delta-r\epsilon)}} }\lesssim\sum_{k=\mu-h}^{\infty}{2^{skq}\big( \mathcal{M}_{t}^{k,\epsilon}f_k(x)\big)^q }.
\end{equation*}
Lemma \ref{maximal1} finally yields
\begin{equation*}
(\ref{adad})\lesssim_{h} \sup_{R\in\mathcal{D}_{\mu}}{\Big(\frac{1}{|R|}\int_R{\sum_{k=\max{(0,\mu-h)}}^{\infty}{2^{skq}|f_k(x)|^q}}dx  \Big)^{1/q}}.
\end{equation*}

\qedhere
\end{proof}

\section{Proof of Positive results}\label{positiveresult}

\subsection*{Paradifferential technique}

The idea of our proof is based on the paradifferential technique, used by Bony \cite{Bo}, Bourdaud \cite{Bou1}, \cite{Bou2}, Marschall \cite{Ma1}, Runst \cite{Ru}, and Johnsen \cite{Jo}, \cite{Jo1}.
For $j,k\geq 0$ let 
\begin{equation}\label{defpse}
a_{j,k}(x,\xi):= \phi_j\ast a(\cdot,\xi)(x)\widehat{\phi_k}(\xi) 
\end{equation}
where we use $\Phi$, instead of $\phi_0$ when $j=0$ or $k=0$.
Then $T_{[a]}f$ can be written as 
\begin{equation*}
T_{[a]}f=\mathfrak{S}_{[a]}^{near}f+\mathfrak{S}_{[a]}^{far}f
\end{equation*}
where 
\begin{equation}\label{kkll}
\mathfrak{S}_{[a]}^{near}f:=\sum_{k,j:|j-k|\leq 2}{T_{[a_{j,k}]}f}
\end{equation}
\begin{equation}\label{llkk}
\mathfrak{S}_{[a]}^{far}f:=\sum_{k,j:|j-k|\geq 3}{T_{[a_{j,k}]}f}.
\end{equation}

It is known in \cite[Proposition 5.6, Theorem 6.1]{Jo3} that $\mathfrak{S}_{[a]}^{far}$ satisfies H\"ormander's twisted diagonal condition (\ref{further}), is well defined on $S'$, and maps $S'$ into itself. Moreover, the adjoints also belong to $Op\mathcal{S}_{1,1}^{m}$.

\subsection{Boundedness of $\mathfrak{S}_{[a]}^{far}$}
Let $0<q<\infty$, $m\in\mathbb{R}$, and $s\in\mathbb{R}$.
We will prove that for a fixed $\mu\in\mathbb{N}$
\begin{align}\label{mainclaim}
&\sup_{P\in\mathcal{D}_{\mu}}\Big(\frac{1}{|P|}\int_P\sum_{k=\mu}^{\infty}{2^{skq} \big|\Pi_k\mathfrak{S}_{[a]}^{far}f(x) \big|^q}dx \Big)^{1/q}\nonumber\\
&\lesssim \sup_{0\leq k\leq \mu-1}{\Vert 2^{k(s+m)}\Pi_kf \Vert_{L^{\infty}}}+\sup_{R\in\mathcal{D}_{\mu}}\Big(\frac{1}{|R|}\int_R\sum_{k=\mu}^{\infty}{2^{(s+m)kq} \big|\Pi_kf(x) \big|^q} dx\Big)^{1/q}
\end{align} uniformly in $\mu$,
and then the inequality (\ref{mainclaim}) implies 
\begin{equation}\label{maincorollary}
\Vert \mathfrak{S}_{[a]}^{far}f\Vert_{F_{\infty}^{s,q}}\lesssim \Vert f\Vert_{F_{\infty}^{s+m,q}} \quad \text{for arbitrary }s\in\mathbb{R}.
\end{equation}
That is, we do not need the condition $s>{\tau}_q$ for $\mathfrak{S}_{[a]}^{far}$ part in Theorem \ref{mmain} and the independence of $s$ for this part will be also used in the proof of Theorem \ref{bmo} and \ref{s=0}.
It follows, from (\ref{maincorollary}) and the embedding $F_{\infty}^{0,1}\hookrightarrow bmo$ in (\ref{bmoembedding}), that for $s=0$
 \begin{equation*}
 \big\Vert \mathfrak{S}_{[a]}^{far}f\big\Vert_{bmo}\lesssim \Vert f\Vert_{F_{\infty}^{m,1}},
 \end{equation*} which is one part of the proof of Theorem \ref{bmo}.

We fix $\mu\in\mathbb{N}$ and $P\in\mathcal{D}_{\mu}$.
Write
\begin{equation*}
\mathfrak{S}_{[a]}^{far}f=\sum_{k=3}^{\infty}{\sum_{j=0}^{k-3}{T_{[a_{j,k}]}f}}+\sum_{j=3}^{\infty}{\sum_{k=0}^{j-3}{T_{[a_{j,k}]}f}}=:\sum_{k=3}^{\infty}{T_{[b_k]}{f}}+\sum_{j=3}^{\infty}{T_{[c_j]}{f}}
\end{equation*}
and observe that $b_k,c_j \in\mathcal{S}_{1,1}^{m}$ uniformly in $k$ and $j$. 
It is clear that
\begin{equation*}
Supp(\widehat{T_{[b_k]}f})\subset \{\xi:2^{k-2}\leq |\xi|\leq 2^{k+2}\},
\end{equation*} 
\begin{equation}\label{ccccj}
Supp(\widehat{T_{[c_j]}f})\subset \{\xi:2^{j-2}\leq |\xi|\leq 2^{j+2}\}
\end{equation} 
and then, from Lemma \ref{marshall1} with $h=2$, we see that
\begin{align}\label{ddd}
&\Big(\frac{1}{|P|}\int_P{\sum_{k=\mu}^{\infty}{2^{skq}\Big|\Pi_k\Big(\sum_{n=3}^{\infty}{T_{[b_n]}f} \Big)(x) \Big|^q}}dx \Big)^{1/q}\nonumber\\ 
&\lesssim \sup_{R\in\mathcal{D}_{\mu}}{\Big( \dfrac{1}{|R|}\int_R{\sum_{k=\max{(3,\mu-2)}}^{\infty}{2^{skq}\big|T_{[b_k]}f(x) \big|^q}}dx \Big)^{1/q}}.
\end{align} 

For $k\geq 3$ let ${\phi}_k^*:=\phi_{k-1}+\phi_k+\phi_{k+1}$ and define $\Pi_k^*f:=\phi_k^*\ast f$ so that $\Pi_k^*\Pi_k=\Pi_k$.
Then our claim is  that for ${\sigma}>0$ and $k\geq 3$
\begin{equation}\label{first}
\big| T_{[b_k]}f(x)\big| \lesssim_{{\sigma}} 2^{km}\mathfrak{M}_{{\sigma},2^k}\Pi^*_kf(x).
\end{equation}
Then (\ref{first}) proves that, by Lemma \ref{maximal2} with $\sigma>d/q$,  
\begin{align*}
(\ref{ddd})&\lesssim \sup_{R\in\mathcal{D}_{\mu}}\Big( \frac{1}{|R|}\int_R\sum_{k=\max{(3,\mu-2)}}^{\infty}{2^{(s+m)kq}\big|\Pi^*_kf(x)\big|^q}dx\Big)^{1/q} \\
&\lesssim \sup_{0\leq k\leq \mu-1}{\big\Vert 2^{k(s+m)}\Pi_kf\big\Vert_{L^{\infty}}} +\sup_{R\in\mathcal{D}_{\mu}}\Big(\frac{1}{|R|}\int_R\sum_{k=\mu}^{\infty}{2^{(s+m)kq} \big|\Pi_kf(x) \big|^q} dx\Big)^{1/q}
\end{align*} where we used a triangle inequality for $\Pi^*_k=\Pi_{k-1}+\Pi_k+\Pi_{k+1}$.

To see (\ref{first}) let $K_k$ be the kernel of $T_{[b_k]}$, which is defined by
\begin{equation}\label{kernelexpress}
K_k(x,y)=\int_{\mathbb{R}^d}{b_k(x,\xi)e^{2\pi i\langle x-y,\xi\rangle}}d\xi.
\end{equation}
Note that 
\begin{equation*}
T_{[b_k]}f(x)=\int_{\mathbb{R}^d}{K_k(x,y)\Pi_k^*f(y)}dy
\end{equation*} where $\Pi_k^*f\in\ C^{\infty}$  and $K_k$ does not have any singularities because the integral in (\ref{kernelexpress}) is, in fact, taken over an annulus $\{\xi:|\xi|\approx 2^k\}$ from the definition of $b_k$.
We first observe that 
\begin{equation}\label{444}
\big|K_k(x,y)\big|\leq \int_{|\xi|\approx 2^k}{|b_k(x,\xi)|}d\xi\lesssim 2^{k(m+d)}.
\end{equation}
 Moreover, by using integration by parts, for $|x-y|\geq 2^{-k}$ one has
\begin{equation}\label{555}
|K_k(x,y)|   \lesssim_M 2^{k(m+d)}\frac{1}{(2^k|x-y|)^M}
\end{equation} for all $M>0$. Combining (\ref{444}) and (\ref{555}) together we have a size estimate
\begin{equation*}
|K_k(x,y)|   \lesssim_M 2^{k(m+d)}\frac{1}{(1+2^k|x-y|)^M}.
\end{equation*}
We choose $M>{\sigma}+d$ and then 
\begin{equation*}
\big| T_{[b_k]}f(x)\big|\lesssim_M 2^{km}\mathfrak{M}_{\sigma,2^k}\Pi_k^*f(x)\int_{\mathbb{R}^d}{\frac{2^{kd}}{\big(1+2^k|x-y|\big)^{M-\sigma}}}dy\lesssim 2^{km}\mathfrak{M}_{\sigma,2^k}\Pi_k^*f(x),
\end{equation*} which proves (\ref{first}).

Similarly, by Lemma \ref{marshall1} with (\ref{ccccj}) one has
\begin{align}\label{eee}
&\Big(\frac{1}{|P|}\int_P{\sum_{k=\mu}^{\infty}{2^{skq}\Big|\Pi_k\Big(\sum_{j=3}^{\infty}{T_{[c_j]}f} \Big)(x) \Big|^q}}dx \Big)^{1/q}\nonumber\\ 
&\lesssim  \sup_{R\in\mathcal{D}_{\mu}}{\Big( \dfrac{1}{|R|}\int_R{\sum_{j=\max{(3,\mu-2)}}^{\infty}{2^{sjq}\big|T_{[c_j]}f(x) \big|^q}}dx \Big)^{1/q}}.
\end{align} 
For the estimation of this part we claim that for each $j\geq 3$ and all $N,\sigma>0$
\begin{equation}\label{second}
|T_{[a_{j,k}]}f(x)| \lesssim_{N,\sigma} 2^{km}2^{-(j-k)N}\mathfrak{M}_{{\sigma},2^k}\Pi^*_kf(x)
\end{equation}  and then  (\ref{eee}) is less than a constant times
\begin{equation*}
\sup_{R\in\mathcal{D}_{\mu}}{\Big(\frac{1}{|R|}\int_R{\sum_{j=\max{(3,\mu-2)}}^{\infty}{2^{sjq}\Big(    \sum_{k=0}^{j-3}{ 2^{km}2^{-(j-k)N}\mathfrak{M}_{{\sigma},2^k}\Pi^*_kf(x)   }\Big)^q}}dx\Big)^{1/q}}.
\end{equation*} 
We choose $\sigma>d/q$, $N>s$ and  $\epsilon>0$ satisfying $s+\epsilon<N$.
Using H\"older's inequality if $q>1$ or embedding $l^q\hookrightarrow l^1$ if $q\leq 1$ this expression is bounded by a constant times
\begin{equation*}
\sup_{R\in\mathcal{D}_{\mu}}\Big(\frac{1}{|R|}\int_R{\sum_{j=\max{(3,\mu-2)}}^{\infty}{2^{{(s+\epsilon)}jq}    \sum_{k=0}^{j-3}{ 2^{kq(m-\epsilon)}2^{-(j-k)Nq}\big(\mathfrak{M}_{{\sigma},2^k}\Pi^*_kf(x) \big)^q  }}}dx\Big)^{1/q}
\end{equation*}
and we exchange the sums over $j$ and over $k$, and split it by using
\begin{equation*}
\sum_{k=\max{(0,\mu-5)}}^{\infty}\sum_{j=k+3}^{\infty}+\sum_{k=0}^{\mu-4}\sum_{j=\mu-2}^{\infty}
\end{equation*} where  we consider only the first part if $\mu\leq 3$.
Then the condition $N-\epsilon-s>0$ and Lemma \ref{maximal2} (with $\sigma>d/q$) prove that the term corresponding to the first sum is controlled by a constant times
\begin{align*}
& \sup_{R\in\mathcal{D}_{\mu}}\Big(\frac{1}{|R|}\int_R{\sum_{k=\max{(0,\mu-5)}}^{\infty}{2^{(m+s)kq}\big| \Pi^*_kf(x)\big|^q}}dx \Big)^{1/q}\\
&\lesssim\sup_{0\leq k\leq \mu-1}{\big\Vert 2^{k(s+m)}\Pi_kf\big\Vert_{L^{\infty}}} +\sup_{R\in\mathcal{D}_{\mu}}\Big(\frac{1}{|R|}\int_R\sum_{k=\mu}^{\infty}{2^{(s+m)kq} \big|\Pi_kf(x) \big|^q} \Big)^{1/q}
\end{align*} uniformly in $\mu$ (where we applied $\Pi_k^*=\Pi_{k-1}+\Pi_k+\Pi_{k+1}$).
Moreover, if $\mu\geq 4$ then the other part is bounded by a constant times
\begin{align*}
&\sup_{R\in\mathcal{D}_{\mu}}{\Big( \frac{1}{|R|}\int_R{\sum_{k=0}^{\mu-4}2^{(k-\mu)(N-\epsilon-s)q}2^{k(m+s)q}\big( \mathfrak{M}_{\sigma,2^k}\Pi^*_kf(x)\big)^q}dx\Big)^{1/q}}\\
&\lesssim \sup_{0\leq k\leq \mu-3}\big\Vert 2^{k(m+s)}\Pi_kf\big\Vert_{L^{\infty}}\leq \sup_{0\leq k\leq \mu-1}\big\Vert 2^{k(m+s)}\Pi_kf\big\Vert_{L^{\infty}}
\end{align*} uniformly in $\mu$, because $N-\epsilon-s>0$.

To see the pointwise estimate (\ref{second}) we also use a size estimate of the kernel of $T_{[a_{j,k}]}$.
Let $W_{j,k}(x,y)$ be the kernels of $T_{[a_{j,k}]}$. Similar to $K_k$ in (\ref{kernelexpress}), $W_{j,k}$ does not have any singularities and it acts on $C^{\infty}$-function $\Pi_k^*f$.
Due to moment conditions $\int_{\mathbb{R}^d}{x^{\alpha}\phi_j(x)}dx=0$ for $j\geq 1$ and any multi-indicies $\alpha$, one has, for any $N\geq 1$,
\begin{align*}
W_{j,k}(x,y)&=\int_{\mathbb{R}^d}{\Big(\int_{\mathbb{R}^d}{\phi_j(z)a(x-z,\xi)}dz\Big) \widehat{\phi_k}(\xi)e^{2\pi i\langle x-y,\xi\rangle}}d\xi\\
&=\int_{\mathbb{R}^d}{\Big[\int_{\mathbb{R}^d}{\phi_j(z)\Big(a(x-z,\xi)-\sum_{|\alpha|\leq N-1}{\frac{1}{\alpha !}\partial_x^{\alpha}a(x,\xi)(-z)^{\alpha}}\Big)}dz\Big] \widehat{\phi_k}(\xi)e^{2\pi i\langle x-y,\xi\rangle}}d\xi
\end{align*}
and then use Taylor's theorem to obtain
\begin{equation*}
W_{j,k}(x,y)=\sum_{|\alpha|=N}{\frac{1}{\alpha !}\int_0^1{(1-t)^{N-1}\Big(\int_{\mathbb{R}^d}{\int_{\mathbb{R}^d}{\phi_j(z)(-z)^{\alpha}\partial_x^{\alpha}a(x-tz,\xi)\widehat{\phi_k}(\xi)e^{2\pi i\langle x-y,\xi\rangle}}dz}d\xi \Big)}dt}.
\end{equation*}
Then for $j\geq 3$, one has
\begin{equation*}
|W_{j,k}(x,y)|\lesssim_N \int_{\mathbb{R}^d}{\int_{\mathbb{R}^d}{|\phi_j(z)||z|^N\big(1+|\xi|\big)^{m+N}|\widehat{\phi_k}(\xi)|}dz}d\xi\lesssim 2^{k(m+d)}2^{-(j-k)N}.
\end{equation*}
Furthermore, for $|x-y|\geq 2^{-k}$ it follows that 
\begin{equation*}
|W_{j,k}(x,y)|\lesssim_{N,M} 2^{k(m+d)}2^{-(j-k)N}\frac{1}{(2^k|x-y|)^M}
\end{equation*}
by integration by parts in $\xi$.
 These yield
\begin{equation*}
|W_{j,k}(x,y)|\lesssim 2^{k(m+d)}2^{-(j-k)N}\frac{1}{(1+2^k|x-y|)^M}
\end{equation*}
and by using the idea in the proof of (\ref{first}) one obtains (\ref{second}). 
This completes the proof of (\ref{mainclaim}).

\subsection{Boundedness of $\mathfrak{S}_{[a]}^{near}$}

For $k\geq 0$ let 
\begin{equation*}
 d_k(x,\xi):=\sum_{j=k-2}^{k+2}{a_{j,k}(x,\xi)}
 \end{equation*} assuming $a_{j,k}(x,\xi)=0$ for $j<0$.
Then $d_k\in \mathcal{S}_{1,1}^{m}$ uniformly in $k$, and 
\begin{equation*}
\mathfrak{S}_{[a]}^{near}f=\sum_{k=0}^{\infty}{T_{[d_k]}f}.
\end{equation*}
By using the argument in the proof of (\ref{first}) we establish the pointwise estimate
\begin{equation}\label{thirdkernelest}
\big| T_{[d_k]}f(x) \big|\lesssim_{\sigma} 2^{km}\mathfrak{M}_{{\sigma},2^k}\Pi^*_kf(x)
\end{equation} for ${\sigma}>d/q$.

\subsubsection{\textbf{Proof of Theorem \ref{mmain}}}\label{mmmn}
Assume $s>\tau_q$.
Since the support of  $Supp(\widehat{T_{[d_k]}f})\subset \{\xi:|\xi|\leq 2^{k+3}\}$, by using Lemma \ref{marshall2} with the assumption $s>\tau_q$, one has 
\begin{equation*}
\Big(\frac{1}{|P|}\int_P{\sum_{k=\mu}^{\infty}{2^{skq}\big|\Pi_k\mathfrak{S}_{[a]}^{near}f(x)\big|^q}}dx \Big)^{1/q} \lesssim \sup_{R\in\mathcal{D}_{\mu}}{\Big(  \frac{1}{|R|}\int_R{\sum_{k=\max{(0,\mu-3)}}^{\infty}{2^{skq}\big|T_{[d_k]}f(x)\big|^q}}dx \Big)^{1/q}}.
\end{equation*}
Then Lemma \ref{maximal2} and triangle inequality for $\Pi_k^*=\Pi_{k-1}+\Pi_k+\Pi_{k+1}$ yield the bound
\begin{align*}
&\sup_{R\in\mathcal{D}_{\mu}}{\Big( \frac{1}{|R|}\int_R{\sum_{k=\max{(0,\mu-4)}}^{\infty}{2^{(s+m)kq}\big|\Pi_kf(x) \big|^q}}dx \Big)^{1/q}}\\
&\lesssim \sup_{0\leq k\leq \mu-1}\big\Vert 2^{k(m+s)}\Pi_kf\big\Vert_{L^{\infty}}+\sup_{R\in\mathcal{D}_{\mu}}\Big(\frac{1}{|R|}\int_R{\sum_{k=\mu}^{\infty}{2^{(m+s)kq}\big| \Pi_kf(x)\big|^q}}dx \Big)^{1/q}
\end{align*} and this completes the proof.

\subsubsection{\textbf{Proof of Theorem \ref{bmo}}}\label{bmosection}
We shall prove
\begin{equation*}
\Vert \mathfrak{S}_{[a]}^{near}f \Vert_{bmo}\lesssim \Vert f\Vert_{F_{\infty}^{m,1}}.
\end{equation*}
According to the definition of $bmo$ in \cite[2.2.2]{Tr}, $\Vert  \mathfrak{S}_{[a]}^{near}f \Vert_{bmo} $ is the sum of
\begin{equation}\label{l(Q)>1}
 \sup_{l(Q)\geq 1}{\frac{1}{|Q|}\int_Q{\big|\mathfrak{S}_{[a]}^{near}f(x)\big|}dx},
 \end{equation}
 \begin{equation}\label{l(Q)<1}
 \sup_{l(Q)< 1}{\frac{1}{|Q|}\int_Q{\big|\mathfrak{S}_{[a]}^{near}f(x)-\big( \mathfrak{S}_{[a]}^{near}f \big)_Q  \big|}dx}
\end{equation} where $(\mathfrak{S}_{[a]}^{near}f)_Q$ is the average of $ \mathfrak{S}_{[a]}^{near}f$ over $Q$ and $\sup_{l(Q)\geq 1}$ means the supremum over all cubes (not necessarily dyadic cubes) satisfying $l(Q)\geq 1$ (and similarly, $ \sup_{l(Q)<1}$ is defined).

By using (\ref{thirdkernelest}) one obtains that for $l(Q)\geq 1$ and $\sigma>d$
\begin{align*}
&\frac{1}{|Q|}\int_Q{\big| \mathfrak{S}_{[a]}^{near}f(x)\big|}dx \lesssim \frac{1}{|Q|}\int_Q{\sum_{k=0}^{\infty}{2^{km}\mathfrak{M}_{\sigma,2^k}\Pi^*_kf(x)}}dx\\
&\lesssim \sup_{0\leq k\leq 2}{\big\Vert 2^{km}\mathfrak{M}_{\sigma,2^k}\Pi^*_kf \big\Vert_{L^{\infty}}}+\frac{1}{|Q|}\int_Q{\sum_{k=3}^{\infty}{2^{km}\mathfrak{M}_{\sigma,2^k}\Pi^*_kf(x)}}dx
\end{align*}

By the $L^{\infty}$ boundedness of $\mathfrak{M}_{\sigma,2^k}$, triangle inequality, and embedding $F_{\infty}^{m,1}\hookrightarrow F_{\infty}^{m,\infty}$ the supremum part is less than a constant times $\Vert f\Vert_{F_{\infty}^{m,1}}$.

For each $P\in \mathcal{D}$ let $3P$ denote a dilate of $P$, defined as in Section \ref{statementsection}.
Let $\mathcal{A}_Q$ be the family of all cubes in $\mathcal{D}_2$ which belong to $Q$. Then we observe that the number of cubes in $\mathcal{A}_Q$ is at most $ \big(2^2l(Q)\big)^d$ and $Q\subset \bigcup_{P\in\mathcal{A}_Q}{(3P)}$.
Therefore
\begin{align*}
\frac{1}{|Q|}\int_Q{\sum_{k=3}^{\infty}{2^{km}\mathfrak{M}_{\sigma,2^k}\Pi^*_kf(x)}}dx &\lesssim \frac{1}{|Q|}\sum_{P\in\mathcal{A}_Q}{\int_{3P}{\sum_{k=3}^{\infty}{2^{km}\mathfrak{M}_{\sigma,2^k}\Pi_k^*f(x)}}dx}\\
&\lesssim \sup_{R\in\mathcal{D}_2}{\frac{1}{|R|}\int_R{\sum_{k=3}^{\infty}{2^{km}\mathfrak{M}_{\sigma,2^k}\Pi_k^*f(x)}}dx}.
\end{align*}
Now by using Lemma \ref{maximal2} and $\Pi_k^*=\Pi_{k-1}+\Pi_k+\Pi_{k+1}$, this is bounded by a constant times 
\begin{equation*}
\sup_{R\in\mathcal{D}_2}\frac{1}{|R|}\int_R\sum_{k=2}^{\infty}{2^{km}\big|\Pi_kf(x) \big|}dx\lesssim \Vert f\Vert_{F_{\infty}^{m,1}},
\end{equation*} which completes the proof of $(\ref{l(Q)>1})\lesssim \Vert f\Vert_{F_{\infty}^{m,1}}$.

For the estimation of the other part (\ref{l(Q)<1}) we assume $l(Q)<1$ and choose $\mu\geq 1$ such that $2^{-\mu}\leq l(Q)<2^{-\mu+1}$.
Then 
\begin{align*}
\frac{1}{|Q|}\int_Q{\big|\mathfrak{S}_{[a]}^{near}f(x)-\big( \mathfrak{S}_{[a]}^{near}f \big)_Q  \big|}dx&\leq {  \frac{1}{|Q|}\int_Q{\sum_{k=0}^{\infty}{\Big| T_{[d_k]}f(x)-\frac{1}{|Q|}\int_Q{T_{[d_k]}f(y)}dy  \Big|}}dx  }\\
&\lesssim I+II
\end{align*} where
\begin{equation*}
I:=\big\Vert T_{[d_0]}f\big\Vert_{L^{\infty}}+{  \frac{1}{|Q|}\int_Q{\sum_{k=\mu +1}^{\infty}{\big| T_{[d_k]}f(x)  \big|}}dx  }
\end{equation*}
\begin{equation*}
II:={  \frac{1}{|Q|}\int_Q\frac{1}{|Q|}\int_Q{\sum_{k=1}^{\mu}{\big| T_{[d_k]}f(x)-{T_{[d_k]}f(y)}  \big|}}dxdy  }.
\end{equation*}

Let $\mathcal{B}_{Q}:=\{P\in\mathcal{D}_{\mu}:P\cap Q\not= \phi\}$. Note that $\mathcal{B}_Q$ has at most $4^d$ elements.
Then \begin{align*}
\frac{1}{|Q|}\int_Q{\sum_{k=\mu+1}^{\infty}{\big|T_{[d_k]}f(x)\big|}}dx&\leq \frac{1}{|Q|}\sum_{P\in\mathcal{B}_Q}{\int_P{\sum_{k=\mu+1}^{\infty}{\big|T_{[d_k]}f(x) \big|}}dx}\\
&\lesssim \sup_{P\in\mathcal{D}_{\mu}}\frac{1}{|P|}\int_P{\sum_{k=\mu+1}^{\infty}{\big|T_{[d_k]}f(x)\big|}}dx
\end{align*} and by using (\ref{thirdkernelest}) and Lemma \ref{maximal2} one has  $I\lesssim \Vert f\Vert_{F_{\infty}^{m,1}}$.

In order to estimate $II$,  let $U_k$ be the kernel of $T_{[d_k]}$. Similar to (\ref{kernelexpress}), $U_k$ has no singularities and it acts on $C^{\infty}$-function $\Pi_k^*f$ when $f\in S'$.
Then for $x,y\in Q$
\begin{align*}
\big|  T_{[d_k]}f(x)-T_{[d_k]}f(y)  \big|&=\Big|  \int_{\mathbb{R}^d}{\big( U_k(x,z)-U_k(y,z) \big)\Pi^*_kf(z)}dz   \Big|\\
      &\leq \big\Vert \Pi^*_kf \big\Vert_{L^{\infty}}  \int_{\mathbb{R}^d}{\big| U_k(x,z)-U_k(y,z)  \big|}dz.
\end{align*}
We claim that for $1\leq k\leq \mu$
\begin{equation}\label{bvc}
\int_{\mathbb{R}^d}{\big| U_k(x,z)-U_k(y,z)  \big|}dz\lesssim 2^{km}2^{-(\mu-k)}.
\end{equation}
If (\ref{bvc}) holds then
\begin{equation*}
II\lesssim \Big(\sup_{k\geq 1}{2^{km}\big\Vert \Pi^*_kf\big\Vert_{L^{\infty}}}\Big)\sup_{\mu\geq 1}{\Big(\sum_{k=1}^{\mu}{2^{-(\mu-k)}}\Big)}\lesssim \Vert f\Vert_{F_{\infty}^{m,\infty}}\lesssim \Vert f\Vert_{F_{\infty}^{m,1}}.
\end{equation*}

To see (\ref{bvc}) we write
\begin{align*}
\big|U_k(x,z)-U_k(y,z) \big|&\leq\Big|  \int_{|\xi|\approx 2^k}{\big(d_k(x,\xi)-d_k(y,\xi) \big)e^{2\pi i\langle x-z,\xi \rangle} }d\xi\Big|\\
&\relphantom{=}+\Big| \int_{|\xi|\approx 2^k}{d_k(y,\xi)e^{2\pi i\langle x-z,\xi\rangle}\big(1-e^{2\pi i\langle y-x,\xi\rangle} \big)}d\xi\Big|.
\end{align*}
Then one obtains that for $x,y\in Q$ and $M>0$ 
 \begin{equation*}
 \big|U_k(x,z)-U_k(y,z) \big|\lesssim_M
 \begin{cases}
2^{km}2^{kd}2^kl(Q), 
\quad 
\\
2^{km}2^{kd}2^kl(Q)\big( 2^k|x-z|\big)^{-M},
\quad  |x-y|\geq 2^{-k}.
\end{cases}
 \end{equation*} 
Indeed, the first one follows from the mean value theorem and $|x-y|\leq \sqrt{d}l(Q)$, and the second one  from integration by parts in $\xi$ variable $M$ times and $2^k|x-y|\leq \sqrt{d}2^kl(Q)\lesssim 1$.
It follows that for $x,y\in Q$
\begin{equation*}
\int_{|x-z|\leq 2^{-k}}{\big| U_k(x,z)-U_k(y,z)  \big|}dz  \lesssim  2^{km}2^kl(Q)
\end{equation*} 
and
\begin{equation*}
\int_{|x-z|> 2^{-k}}{\big| U_k(x,z)-U_k(y,z)  \big|}dz \lesssim_M 2^{km}2^kl(Q)
\end{equation*} for $M>d$, which proves (\ref{bvc}).

\subsubsection{\textbf{Proof of Theorem \ref{s=0}}}\label{s=0section}

In the proof of Theorem \ref{mmain} the condition $s>\tau_q$ is used just to apply Lemma \ref{marshall2} for $T_{[d_k]}$ because the Fourier transform of $T_{[d_k]}f$ is supported in a ball. 
We observe that if (\ref{further}) holds, then $T_{[d_k]}f$ has a Fourier support in an annulus so that we apply Lemma \ref{marshall1}, instead of Lemma \ref{marshall2}. Therefore we can drop the restriction on $s$ with (\ref{further}).\\

\section{Proof of Negative results}\label{negativeresult}

As mentioned in Remark \ref{remark2.8} we consider only the case $\min{(p,q)}\leq 1$. 
We will construct  $a\in \mathcal{S}_{1,1}^{m}$ and $f\in F_{p}^{s+m,q}$ of the form 
\begin{equation*}
a(x,\xi)=\sum_{k}{F_k(\xi)G_k(x)}
\end{equation*}
\begin{equation*}
f(x)=\sum_{k}{H_k(x)}
\end{equation*} where $F_k,G_k,H_k\in S$ such that $Supp(F_k), Supp(\widehat{H_k})\subset \{\xi:|\xi|\approx 2^k\}$. 

In fact, it is enough to construct examples for $m=0$ because the case $m\not= 0$ follows immediately by replacing $F_k$ and $H_k$ by $2^{km}F_k$ and $2^{-km}H_k$, respectively.
Therefore we assume $m=0$.

The counterexamples in this section have three different natures based on the condition on $s$.
\begin{itemize}
\item $s\leq d/q-d$ - The main idea is the random construction method by Christ and Seeger \cite{Ch_Se}.
\item $s\leq d/p-d$ - The problem is closely related to the sharpness in local Hardy space $h^p$, $0<p\leq 1$.
Actually, the motivation of the examples below is the problem to construct $f\in S$ and $a\in\mathcal{S}_{1,1}^{m}$ such that $\Vert \Pi_kf\Vert_{h^p}\approx \Vert \Pi_kf\Vert_{L^p}\lesssim 2^{-sk}$ and $\Vert \Pi_kT_{[a]}f\Vert_{h^p}\approx \Vert \Pi_kT_{[a]}f\Vert_{L^p}\gtrsim 2^{-sk}k^{\epsilon}$ for some $\epsilon>0$.  
\item $s\leq 0$ - We follow the argument in Remark \ref{remark2.8}, which is similar to Ching's method.

\end{itemize}

Let $M>0$ be a fixed integer and in what follows let $\Gamma, \Lambda \in S$ satisfy $Supp(\widehat{{\Gamma}})\subset \{\xi : 1<|\xi|<2\}$,  ${\Gamma}(x)\geq 1$ for $|x|\leq 2^{-M+1}$, $Supp(\widehat{{\Lambda}})\subset \{\xi : {9}/{8}<|\xi|<{15}/{8}\}$, $\widehat{{\Lambda}}\geq 0$, and $\widehat{{\Lambda}}(\xi)= 1$ on $\{\xi : {5}/{4}\leq |\xi|\leq {7}/{4}\}$. Then define ${\Gamma}_k:=2^{kd}{\Gamma}(2^k\cdot)$ and  ${\Lambda}_k:=2^{kd}{\Lambda}(2^k\cdot)$.
Let $t_k:=10k$, $e_1:=(1,0,\dots,0)\in\mathbb{Z}^d$ and $\mathbf{v}_k:=(3/2)2^ke_1$.

\subsection*{Construction of pseudo-differential operators}\label{construction}

Define a symbol $a\in\mathcal{S}_{1,1}^{0}$ by
\begin{equation}\label{gggfff}
a(x,\xi):=\sum_{k=1}^{\infty}{\widehat{{\Lambda}}({\xi}/{2^k})e^{2\pi i\langle \mathbf{v}_k, x \rangle}}.
\end{equation}
For a sufficiently large number $L>M$ let $\{g_k\}_{k=M}^{L}$ be a family of Schwartz functions to be chosen later, and  define 
\begin{equation}\label{fffggg}
f_L(x):=\sum_{k=M}^{L}{2^{-{s} {t_k}}{\Lambda}_{{t_k}}\ast \big(g_{{t_k}} e^{-2\pi i\langle \mathbf{v}_{{t_k}},\cdot \rangle}\big)(x)}.
\end{equation}

Then we observe that
\begin{equation}\label{dexpress}
{T_{[a]}}f_L(x)=\sum_{k=1}^{\infty}{{\Lambda}_k\ast f_L(x)e^{2\pi i\langle \mathbf{v}_k,x \rangle}}=\sum_{n=M}^{L}{d_{{t_n}}(x)}
\end{equation}
where
\begin{equation*}
d_{t_n}(x):={2^{-{s} {t_n}}{\Lambda}_{{t_n}}\ast{\Lambda}_{{t_n}}\ast\big(g_{{t_n}}e^{-2\pi i\langle \mathbf{v}_{{t_n}},\cdot \rangle}\big)(x)e^{2\pi i\langle \mathbf{v}_{{t_n}},x \rangle}}.
\end{equation*}
Moreover, the support conditions of $\widehat{\phi_k}$ and $\widehat{{\Lambda}_{t_k}}$ ensure that
\begin{equation*}
\phi_k\ast f_L(x)=
\begin{cases}
2^{-{s} {t_n}}\phi_{{t_n}}\ast{\Lambda}_{{t_n}}\ast\big(g_{{t_n}}e^{-2\pi i\langle \mathbf{v}_{{t_n}},\cdot \rangle}\big)(x), \quad &
\quad k={t_n}, \ M\leq n\leq L,
\\
2^{-{s} {t_n}}\phi_{{t_n}+1}\ast{\Lambda}_{{t_n}}\ast\big(g_{{t_n}}e^{-2\pi i \langle\mathbf{v}_{{t_n}},\cdot \rangle}\big)(x), \quad &
\quad  k={t_n}+1, \ M\leq n\leq L
\\
0, \quad &
\quad  otherwise,
\end{cases},
\end{equation*}
and thus for $p<\infty$ and any ${\sigma} >0$
\begin{align}\label{estbound}
\big\Vert f_L\big\Vert_{F_p^{{s},q}}&\lesssim \Big[\int_{\mathbb{R}^d}{\Big(\sum_{n=M}^{L}{ \big|{\phi_{{t_n}}\ast {\Lambda}_{{t_n}}\ast\big(g_{{t_n}}e^{-2\pi i \langle \mathbf{v}_{{t_n}},\cdot \rangle}\big)(x)}\big|^q}\Big)^{{p}/{q}}}dx\Big]^{{1}/{p}}\nonumber\\
&\relphantom{=}+\Big[\int_{\mathbb{R}^d}{\Big(\sum_{n=M}^{L}{ \big|{\phi_{{t_n}+1}\ast {\Lambda}_{{t_n}}\ast\big(g_{{t_n}}e^{-2\pi i\langle \mathbf{v}_{{t_n}},\cdot \rangle}\big)(x)}\big|^q}\Big)^{{p}/{q}}}dx\Big]^{{1}/{p}}\nonumber\\
&\lesssim_{\sigma} \Big[\int_{\mathbb{R}^d}{\Big(\sum_{n=M}^{L}{\big(\mathfrak{M}_{{\sigma},2^{t_n}}{g_{{t_n}}}(x)\big)^q}\Big)^{{p}/{q}}}dx\Big]^{{1}/{p}}
\end{align}  where the usual modification is applied if $q=\infty$. 

On the other hand, 
 \begin{align}
\big\Vert   {T_{[a]}f_L}  \big\Vert_{F_p^{{s},q}} &=\Big\Vert   \Big(\sum_{l=0}^{\infty}{2^{{s} lq}\Big| \sum_{n=M}^{L}{\phi_l\ast d_{{t_n}}} \Big|^q  }\Big)^{{1}/{q}}\Big\Vert_{L^p}\nonumber\\
  &\gtrsim \Big\Vert   \Big(\sum_{l=M}^{L-M}{2^{{s} {t_l}q}\Big| \sum_{n=M}^{L}{(\phi_{{t_l}}+\phi_{{t_l}+1})\ast d_{{t_n}}} \Big|^q    }\Big)^{{1}/{q}}\Big\Vert_{L^p} \label{lowermain}.
\end{align}
Since $Supp(\widehat{d_{{t_n}}})\subset \{\xi :|\xi|<({27}/{16})2^{{t_n}+1}\}$ and $Supp(\widehat{\phi_{{t_l}}})\subset \{\xi:2^{{t_l}-1}\leq |\xi|\leq 2^{{t_l}+1}\}$, $\phi_{{t_l}}\ast d_{{t_n}}$ vanishes unless $l\leq n$.
Therefore for some $C>0$ one has
\begin{align}
(\ref{lowermain})&= \Big\Vert   \Big(\sum_{l=M}^{L-M}{2^{{s} {t_l}q}\Big| \sum_{n=l}^{L}{(\phi_{{t_l}}+\phi_{{t_l}+1})\ast d_{{t_n}}} \Big|^q    }\Big)^{{1}/{q}}\Big\Vert_{L^p}\nonumber\\
  &\geq C\Big\Vert   \Big(\sum_{l=M}^{L-M}{2^{{s} {t_l}q}\Big| \sum_{n=l+1}^{L}{(\phi_{{t_l}}+\phi_{{t_l}+1})\ast d_{{t_n}}} \Big|^q    }\Big)^{{1}/{q}}\Big\Vert_{L^p}\label{main}\\
  & \relphantom{=} -\Big\Vert   \Big(\sum_{l=M}^{L-M}{2^{{s} {t_l}q}\big| {(\phi_{{t_l}}+\phi_{{t_l}+1})\ast d_{{t_l}}} \big|^q    }\Big)^{{1}/{q}}\Big\Vert_{L^p}\label{error}.
\end{align} 

(\ref{error}) is a minor part and indeed, 
\begin{equation}\label{minorkey}
\big| {(\phi_{{t_l}}+\phi_{{t_l}+1})\ast d_{{t_l}}} \big|\lesssim_{\sigma}\mathfrak{M}_{\sigma,2^{t_l}}d_{t_l}(x)\lesssim 2^{-st_l}\mathfrak{M}_{\sigma,2^{t_l}}g_{t_l}(x)
\end{equation}
and thus (\ref{error}) is  bounded above by a constant times 
\begin{equation*}
\Big[\int_{\mathbb{R}^d}{\Big(\sum_{l=M}^{L}{\big(\mathfrak{M}_{{\sigma},2^{t_l}}{g_{{t_l}}}(x)\big)^q}\Big)^{{p}/{q}}}dx\Big]^{{1}/{p}}.
\end{equation*}
For (\ref{main}) we observe that
if $2^{{t_l}-1}\leq |\xi|\leq 2^{{t_l}+2}$ and $l+1\leq n$ then  
\begin{equation*}
(5/4) 2^{t_n}<\big({3}/{2}-{1}/{2^8}\big)2^{{t_n}}\leq |\xi -\mathbf{v}_{{t_n}}|\leq \big({3}/{2}+{1}/{2^8}\big)2^{{t_n}}<(7/4) 2^{t_n}
\end{equation*}
and for such $\xi$ one has $\widehat{{\Lambda}_{{t_n}}}(\xi-\mathbf{v}_{{t_n}})=1$.
Hence we write 
\begin{equation*}
\widehat{d_{{t_n}}}(\xi)=2^{-{s} {t_n}}\widehat{{\Lambda}_{{t_n}}}(\xi-\mathbf{v}_{{t_n}})\widehat{{\Lambda}_{{t_n}}}(\xi-\mathbf{v}_{{t_n}})\widehat{g_{{t_n}}}(\xi)=2^{-{s} {t_n}}\widehat{g_{{t_n}}}(\xi)
\end{equation*} and this gives  
\begin{equation}\label{doublecondition}
(\phi_{{t_l}}+\phi_{{t_l}+1}) \ast d_{{t_n}}(x)=2^{-{s}  {t_n}}(\phi_{{t_l}}+\phi_{{t_l}+1})\ast g_{{t_n}}(x)
\end{equation} for $l+1\leq n$.
Finally one obatins
\begin{align}\label{dfbound}
(\ref{main})&\approx \Big\Vert   \Big(\sum_{l=M}^{L-M}{2^{{s} {t_l}q}\Big| \sum_{n=l+1}^{L}{2^{-{s} {t_n}}\big(\phi_{{t_l}}+\phi_{{t_l}+1}\big)\ast g_{{t_n}}} \Big|^q    }\Big)^{{1}/{q}}   \Big\Vert_{L^p}\nonumber\\
&\gtrsim \Big\Vert   \Big[\sum_{l=M}^{L-M}{2^{{s} {t_l}q}\Big(\mathfrak{M}_{{\sigma},2^{t_l}}\Big( \sum_{n=l+1}^{L}{2^{-{s} {t_n}}\big(\phi_{{t_l}}+\phi_{{t_l}+1}\big)\ast g_{{t_n}}}\Big) \Big)^q    }\Big]^{{1}/{q}}   \Big\Vert_{L^p}\nonumber\\
&\gtrsim \Big\Vert   \Big(\sum_{l=M}^{L-M}{2^{{s} {t_l}q}\Big| \sum_{n=l+1}^{L}{2^{-{s} {t_n}} {\Gamma}_{{t_l}}\ast \big(\phi_{{t_l}}+\phi_{{t_l}+1}\big)\ast g_{{t_n}}} \Big|^q   }\Big)^{{1}/{q}}   \Big\Vert_{L^p}\nonumber\\
&= \Big\Vert   \Big(\sum_{l=M}^{L-M}{\Big| \sum_{n=1}^{L-l}{2^{-{s} {t_n}} {\Gamma}_{{t_l}}\ast g_{t_{n+l}}} \Big|^q    }\Big)^{{1}/{q}}   \Big\Vert_{L^p}
\end{align} by (\ref{max}) with ${\sigma} >\max{({d}/{p},{d}/{q})}$ and the fact that $\Gamma_{t_l}\ast \big(\phi_{t_l}+\phi_{t_l+1}\big)=\Gamma_{t_l}$.

So far we have established that for $p<\infty$ and $\sigma>0$
\begin{equation}\label{upper}
\big\Vert f_L\big\Vert_{F_p^{{s},q}}\lesssim \Big[\int_{\mathbb{R}^d}{\Big(\sum_{n=M}^{L}{\big(\mathfrak{M}_{{\sigma},2^{t_n}}{g_{{t_n}}}(x)\big)^q}\Big)^{{p}/{q}}}dx\Big]^{{1}/{p}},
\end{equation}  
and for $\sigma>\max{(d/p,d/q)}$
\begin{align}\label{lower}
\big\Vert  {T_{[a]}}f_L \big\Vert_{F_p^{{s},q}}&\geq A \Big\Vert   \Big(\sum_{l=M}^{L-M}{\Big| \sum_{n=1}^{L-l}{2^{-{s} {t_n}} {\Gamma}_{{t_l}}\ast g_{t_{n+l}}} \Big|^q    }\Big)^{{1}/{q}}   \Big\Vert_{L^p} \\
&\relphantom{=}-B \Big[\int_{\mathbb{R}^d}{\Big(\sum_{n=M}^{L}{\big(\mathfrak{M}_{{\sigma},2^{t_n}}{g_{{t_n}}}(x)\big)^q}\Big)^{{p}/{q}}}dx\Big]^{{1}/{p}}\nonumber
\end{align} 
where $A$ and $B$ are some positive numbers.

\subsection{Sharpness of $s>d/q-d$ in $F$-space when $0<q\leq 1$ and $q<p\leq \infty$}
We will prove Theorem \ref{sharpresult1} (1) for $0<q\leq 1$ and $q<p< \infty$ and Theorem \ref{sharpresult2} for $0<q\leq 1$.

\subsubsection{{Proof of Theorem \ref{sharpresult1} (1)}}\label{q<p}

Suppose $0<q\leq 1$, $q<p< \infty$, and $s\leq d/q-d$.
For each $k\in \mathbb{N}_0$, let $\mathcal{Q}(k)$ be the set of all dyadic cubes of side length $2^{-{t_k}-M}$ in $[0,1]^d$ and $ \mathcal{Q}:=\bigcup_{k\in\mathbb{Z}}{\mathcal{Q}(k)}$.
Let $\Omega$ be a probability space with probability measure $\lambda$. 
Let $\{\theta_{Q}\}$ be a family of independent random variables indexed by $Q\in\mathcal{Q}$, each of which takes the value $1$ with probability ${1}/{L}$ and the value $0$ with probability $1-{1}/{L}$.
Consider random functions 
\begin{equation*}
h_{k}^{\omega}(x):=\sum_{Q\in\mathcal{Q}(k)}{\theta_{Q}(\omega)\chi_Q(x)}
\end{equation*} for $\omega\in\Omega$.
Then the following result holds due to Christ and Seeger \cite{Ch_Se}.
\begin{lemma} \cite{Ch_Se}\label{ChSe}
Suppose $0<p,q<\infty$ and ${\sigma}>\max{({d}/{p},{d}/{q})}$.
Then \begin{equation*}
\Big(\int_{\Omega}{\Big\Vert   \Big(    \sum_{k=1}^{L}{\big(\mathfrak{M}_{{\sigma},2^{t_k}}h_{k}^{\omega}\big)^q}           \Big)^{{1}/{q}}       \Big\Vert_{L^p}^p    }d\lambda\Big)^{{1}/{p}}\lesssim_{p,q} 1
\end{equation*}
uniformly in $L$.
\end{lemma}
We note that one of the key idea in the proof of Lemma \ref{ChSe} is  the pointwide estimate
\begin{equation}\label{chse}
\mathfrak{M}_{\sigma,2^{t_k}}h_k^{\omega}(x)\lesssim_\sigma \mathcal{M}_rh_k^{\omega}(x) \quad \text{for}~ \sigma>d/r.
\end{equation}
To obtain lower bounds we  benefit from the Khintchine's inequality.
\begin{lemma}(Khintchine's inequality)\label{kh}
Let  $\{r_n(t)\}_{n=1}^{\infty}$ be Rademacher functions and $\{a_n\}_{n=1}^{\infty}$ be a sequence of complex numbers.
Suppose $0<p<\infty$.
Then \begin{equation*}
\Big( \int_{0}^1{\Big|  \sum_{n=1}^{\infty}{a_nr_n(t)} \Big|^p}dt  \Big)^{{1}/{p}}\approx \Big(\sum_{n=1}^{\infty}{|a_n|^2}\Big)^{{1}/{2}}.
\end{equation*}
\end{lemma}

Let $\{r_{Q}\}$ be a family of Rademacher functions, defined on $[0,1]$, indexed by $Q\in\mathcal{Q}$ and define random functions 
\begin{equation*}
h_{k}^{\omega,t}(x):=\sum_{Q\in\mathcal{Q}(k)}{r_Q(t)\theta_{Q}(\omega)\chi_Q(x)}
\end{equation*} for $\omega\in\Omega, t\in [0,1]$.
Then by using the pointwise estimate
$\mathfrak{M}_{{\sigma},2^{t_k}}h_{k}^{\omega,t}(x) \leq \mathfrak{M}_{{\sigma},2^{t_k}}h_{k}^{\omega}(x)$ and Lemma \ref{ChSe} we establish 
\begin{equation}\label{corr}
\Big(\int_{0}^1{\int_{\Omega}{\Big\Vert   \Big(    \sum_{k=1}^{L}{(\mathfrak{M}_{{\sigma},2^{t_k}}h_{k}^{\omega,t})^q}           \Big)^{{1}/{q}}       \Big\Vert_{L^p}^p    }d\lambda}dt\Big)^{{1}/{p}}\lesssim_{p,q} 1
\end{equation}
for $0<p,q<\infty$ and ${\sigma}>\max{({d}/{p},{d}/{q})}$.

Now we apply (\ref{upper}) with $g_{{t_k}}=h_k^{\omega,t}$ and (\ref{corr}) to obtain
\begin{equation*}
\Big(\int_{0}^1{\int_{\Omega}{\big\Vert f_L^{\omega,t}\big\Vert_{F_p^{{s},q}}^p}d\lambda}dt\Big)^{{1}/{p}} \lesssim \Big(\int_{0}^1{\int_{\Omega}{\Big\Vert   \Big(    \sum_{k=M}^{L}{\big(\mathfrak{M}_{{\sigma},2^{t_k}}h_{k}^{\omega,t}\big)^q}           \Big)^{{1}/{q}}       \Big\Vert_{L^p}^p    }d\lambda}dt\Big)^{{1}/{p}}\lesssim 1
\end{equation*}  uniformly in $L$.

Thus, due to (\ref{lower}) it suffices to show that for ${s}+d-{d}/{q}\leq 0$ 
\begin{align}
&\Big(\int_{0}^1{\int_{\Omega}{\Big\Vert   \Big(\sum_{l=M}^{L-M}{\big| \sum_{n=1}^{L-l}{2^{-{s} {t_n}} {\Gamma}_{{t_l}}\ast h_{{n+l}}^{\omega,t}} \big|^q    }\Big)^{{1}/{q}}   \Big\Vert_{L^p}^p}d\lambda}dt\Big)^{{1}/{p}}\nonumber\\
&\gtrsim L^{-({s}+d-{d}/{q})/(2d)}\big(\log{L}\big)^{{1}/{2}}\label{toshow}.
\end{align}

By H\"older's inequality with $p/q>1$ and Khintchine's inequality we see that
\begin{align*}
&\Big(\int_{0}^{1}{\int_{\Omega}{\Big\Vert   \Big(\sum_{l=M}^{L-M}{\Big| \sum_{n=1}^{L-l}{2^{-{s} {t_n}} {\Gamma}_{{t_l}}\ast h_{{n+l}}^{\omega,t}} \Big|^q    }\Big)^{{1}/{q}}   \Big\Vert_{L^p}^p}d\lambda}dt\Big)^{{1}/{p}}\nonumber\\
&\geq \Big(\int_{0}^{1}{\int_{\Omega}{\int_{[0,1]^d}{\sum_{l=M}^{L-M}{\Big| \sum_{n=1}^{L-l}{2^{-{s} {t_n}} {\Gamma}_{{t_l}}\ast h_{{n+l}}^{\omega,t}(x)} \Big|^q    } }dx      }d\lambda}dt\Big)^{{1}/{q}}\nonumber\\
&= \Big(\int_{[0,1]^d}{\int_{\Omega}{\sum_{l=M}^{L-M}{\int_{0}^{1}{\Big| \sum_{n=1}^{L-l}{   \sum_{Q\in\mathcal{Q}(n+l)}{2^{-{s} {t_n}}r_Q(t)\theta_Q(\omega) {\Gamma}_{{t_l}}\ast \chi_Q(x)}} \Big|^q    }dt}      }d\lambda}dx\Big)^{{1}/{q}}\nonumber\\
&\gtrsim \Big[  \int_{[0,1]^d}{   \int_{\Omega}{        \sum_{l=M}^{L-M}{ \Big(  \sum_{n=1}^{L-l}{\sum_{Q\in\mathcal{Q}(n+l)}{2^{-2{s} {t_n}}\theta_Q(\omega)\big|{\Gamma}_{{t_l}}\ast \chi_Q(x)\big|^2   }} \Big)^{{q}/{2}}       }      }d\lambda      }dx       \Big]^{{1}/{q}}.\label{est}
\end{align*} 
For each $P\in\mathcal{Q}(l)$ and $n\geq 0$ let $\mathcal{V}_n(l,P)=\{Q\in\mathcal{Q}(l+n):Q\subset P\}.$
Then  the last expression is bounded below by \begin{equation}
\Big[ \sum_{l=M}^{L-M}{  \sum_{P\in\mathcal{Q}(l)}{ \int_{P}{     \int_{\Omega}{ \Big(   \sum_{n=1}^{L-l}{ 2^{-2{s} {t_n}}\sum_{Q\in\mathcal{V}_n(l,P)}{\theta_Q(\omega)\big|{\Gamma}_{{t_l}}\ast\chi_Q(x)\big|^2}       }             \Big)^{{q}/{2}}             }d\lambda}dx}             }   \Big]^{{1}/{q}}\label{3.3}.
\end{equation}
For $x\in P$ and $Q\in \mathcal{V}_n(l,P)$ we have ${\Gamma}_{{t_l}}\ast\chi_Q(x)\geq 2^{-{t_n}d}$ because ${\Gamma}(x)\geq 1$ for $|x|\leq 2^{-M}$.
This yields that (\ref{3.3}) is greater than 
\begin{equation}\label{mainest}
\Big[ \sum_{l=M}^{L-M}{ 2^{-{t_l}d} \sum_{P\in\mathcal{Q}(l)}{ \int_{\Omega}{   \Big(  \sum_{n=1}^{L-l}{2^{-2({s}+d) {t_n}}\sum_{Q\in\mathcal{V}_n(l,P)}{\theta_Q(\omega)}        }    \Big)^{{q}/{2}}        }d\lambda           }       }   \Big]^{{1}/{q}}
\end{equation}
 and, by Minkowski's inequality   with $2/q>1$
\begin{align}\label{min}
&\sum_{P\in\mathcal{Q}(l)}{\int_{\Omega}{ \Big(  \sum_{n=1}^{L-l}{  2^{-2({s}+d){t_n}}\sum_{Q\in\mathcal{V}_n(l,P)}{\theta_Q(\omega)}   }   \Big)^{{q}/{2}}   }d\lambda  }\nonumber  \\ 
&\geq \Big(  \sum_{n=1}^{L-l}{2^{-2({s}+d){t_n}}  \Big[   \sum_{P\in\mathcal{Q}(l)}{ \int_{\Omega}{ \Big(\sum_{Q\in\mathcal{V}_n(l,P)}{\theta_Q(\omega)}\Big)^{{q}/{2}} }d\lambda  }    \Big]^{{2}/{q}}     }   \Big)^{{q}/{2}}. 
\end{align}
For $P\in\mathcal{Q}(l)$ and $R\in\mathcal{V}_n(l,P)$, let $\Omega(P,R,l,n)$ be the event that $\theta_R(\omega)=1$, but $\theta_{R'}(\omega)=0$ for $R'\in \mathcal{V}_n(l,P)\setminus \{R\}$. We observe that the probability of this event is $\lambda(\Omega(P,R,l,n))\geq{1}/{L}\big(1-{1}/{L}\big)^{2^{{t_n}d}}$.
Therefore \begin{align*}
 \int_{\Omega}{ \Big(\sum_{Q\in\mathcal{V}_n(l,P)}{\theta_Q(\omega)}\Big)^{{q}/{2}} }d\lambda  &\geq \sum_{R\in\mathcal{V}_{n}(l,P)}{\int_{\Omega(P,R,l,n)}{ \Big(\sum_{Q\in\mathcal{V}_n(l,P)}{\theta_Q(\omega)}\Big)^{{q}/{2}}}d\lambda} \\
 &= \sum_{R\in\mathcal{V}_n(l,P)}{\lambda(\Omega(P,R,l,n))   }\geq2^{{t_n}d}\frac{1}{L}\Big(1-\frac{1}{L}\Big)^{2^{{t_n}d}}
\end{align*}
since $|\mathcal{V}_n(l,P)|=2^{t_nd}$, and this implies
\begin{equation*}
(\ref{min}) \geq 2^{{t_l}d}\frac{1}{L} \Big(  \sum_{n=1}^{L-l}{2^{-2({s}+d-{d}/{q}){t_n}}\Big(1-\frac{1}{L}\Big)^{({2}/{q})2^{{t_n}d}}}   \Big)^{{q}/{2}}.
\end{equation*}
Finally one obtains 
\begin{align*}
(\ref{mainest})&\geq\Big[  \frac{1}{L}\sum_{l=M}^{L-M}{ \Big(   \sum_{n=1}^{L-l}{2^{-2({s}+d-{d}/{q}){t_n}}\Big(1-\frac{1}{L}\Big)^{({2}/{q}){2^{{t_n}d}}}}     \Big)^{{q}/{2}}  }  \Big]^{{1}/{q}}\\
&\geq \Big[  \frac{1}{L}\sum_{l=\lfloor L/3\rfloor}^{\lfloor L/2\rfloor}{ \Big(   \sum_{n=1}^{\lfloor (1/10d)\log_2{L}\rfloor}{2^{-2({s}+d-{d}/{q}){t_n}}\Big(1-\frac{1}{L}\Big)^{({2}/{q}){L}}}     \Big)^{{q}/{2}}  }  \Big]^{{1}/{q}}\\
&\gtrsim L^{-({s}+d-{d}/{q})/(2d)}\big(\log{L}\big)^{{1}/{2}}
\end{align*} 
for sufficiently large $L>0$, and this proves (\ref{toshow}).

\subsubsection{{Proof of Theorem  \ref{sharpresult2}}}
Suppose $\mu\in \mathbb{N}$, $M>\mu$, $0<q\leq 1$, and $s\leq  d/q-d$. Let $L$ be a sufficiently large integer.
For each $k\in \mathbb{N}_0$ let $\mathcal{R}^{\mu}(k)$ be the set of all dyadic cubes of side length $2^{-{t_k}-M}$ in $[0,2^{-\mu}]^d$ and $ \mathcal{R}^{\mu}:=\bigcup_{k\in\mathbb{Z}}{\mathcal{R}^{\mu}(k)}$.
Let $\Omega$ be a probability space with probability measure $\lambda$. 
Let $\{\theta_{Q}\}$ be a family of independent random variables indexed by $Q\in\mathcal{R}^{\mu}$, each of which takes the value $1$ with probability ${1}/{L}$ and the value $0$ with probability $1-{1}/{L}$.
Let $\{r_{Q}\}$ be a family of  Rademacher functions, defined on $[0,1]$, indexed by $Q\in\mathcal{R}^{\mu}$.  For $\omega\in\Omega, t\in [0,1]$ define random functions 
\begin{equation*} 
h_{k}^{\omega,t,\mu}(x):=\sum_{Q\in\mathcal{R}^{\mu}(k)}{r_Q(t)\theta_{Q}(\omega)\chi_Q(x)}.
\end{equation*} 
and
\begin{equation*}
f_L^{\omega,t,\mu}(x):=\sum_{k=M}^{L}{2^{-st_k}\Lambda_{t_k}\ast\big(h_{k}^{\omega,t,\mu}e^{-2\pi i\langle   \mathbf{v}_k,\cdot\rangle} \big)(x)},
\end{equation*} which is of the form (\ref{fffggg}),
 and let $a\in\mathcal{S}_{1,1}^{0}$ be defined as in (\ref{gggfff}).
Then our claim is that
\begin{equation}\label{upperbd}
\Big[ \int_0^1\int_{\Omega}\sup_{R\in\mathcal{D}_{\mu}}\Big(\frac{1}{|R|}\int_{R}\sum_{k=\mu}^{\infty}{2^{skq}\big| \Pi_k f_L^{\omega,t,\mu}(x)\big|^q}dx\Big)     d\lambda dt\Big]^{1/q}\lesssim 1
\end{equation} uniformly in $L$ and for all large $L$
\begin{align}\label{lowerbd}
&\Big[ \int_0^1\int_{\Omega}\sup_{R\in\mathcal{D}_{\mu}}\Big(\frac{1}{|R|}\int_{R}\sum_{k=\mu}^{\infty}{2^{skq}\big| \Pi_k T_{[a]}f_L^{\omega,t,\mu}(x)\big|^q}dx\Big)     d\lambda dt\Big]^{1/q}\nonumber\\
&\gtrsim L^{-(s+d-d/q)/2d}\big( \log{L}\big)^{1/2}.
\end{align}

We observe that the left hand side of (\ref{upperbd}) is less than a constant times
\begin{equation*}
\Big[ \int_0^1\int_{\Omega}\sup_{R\in\mathcal{D}_{\mu}}\Big(\frac{1}{|R|}\int_{R}\sum_{k=M}^{L}{\big( \mathfrak{M}_{\sigma,2^{t_k}}h_k^{\omega,t,\mu}(x)\big)^q}dx\Big)     d\lambda dt\Big]^{1/q}
\end{equation*} by the argument that led to the inequaltiy (\ref{estbound}). Then by (\ref{chse}) and $L^q$-boundedness of  $\mathcal{M}_r$ for $r<q$,
it is dominated by a constant times
\begin{equation*}
\Big( \int_0^1{\int_{\Omega}{2^{\mu d}\sum_{k=M}^{L}{\big\Vert \mathcal{M}_rh_k^{\omega,t,\mu}\big\Vert_{L^q}^q}}d\lambda}dt\Big)^{1/q}\lesssim \Big(2^{\mu d}\sum_{k=M}^{L}{2^{-t_kd}2^{-Md}\sum_{Q\in\mathcal{R}^{\mu}(k)}{\lambda\big(\{\theta_Q=1\}\big)}} \Big)^{1/q}.
\end{equation*}
Then (\ref{upperbd}) follows from $|\mathcal{R}^{\mu}(k)|=2^{-\mu d}2^{t_kd}2^{Md}$ and $\lambda\big(\{\theta_Q=1\}\big)=1/L$.

Now we write, as in (\ref{dexpress}),
\begin{equation*}
T_{[a]}f_L^{\omega,t,\mu}(x)=\sum_{k=M}^{L}{d_{t_k}(x)}
\end{equation*}
where 
\begin{equation*}
d_{t_k}(x):=2^{-st_j}\Lambda_{t_k}\ast\Lambda_{t_k}\ast\big( h_k^{\omega,t,\mu}e^{-2\pi i\langle \mathbf{v}_k,\cdot\rangle}\big)(x)e^{2\pi i\langle \mathbf{v}_k,x\rangle}.
\end{equation*}
Then the left hand side of the inequality (\ref{lowerbd}) is bounded below by a constant times
\begin{align}
& \Big[ \int_0^1\int_{\Omega}\sup_{R\in\mathcal{D}_{\mu}}\Big(\frac{1}{|R|}\int_{R}\sum_{k=M}^{L}{2^{st_kq}\Big| \sum_{n=k}^{L}(\phi_{t_k}+\phi_{t_k+1})\ast d_{t_n}(x)\Big|^q}dx\Big)     d\lambda dt\Big]^{1/q}\nonumber\\
&\geq C\Big[ \int_0^1\int_{\Omega}\sup_{R\in\mathcal{D}_{\mu}}\Big(\frac{1}{|R|}\int_{R}\sum_{k=M}^{L}{2^{st_kq}\Big| \sum_{n=k+1}^{L}(\phi_{t_k}+\phi_{t_k+1})\ast d_{t_n}(x)\Big|^q}dx\Big)     d\lambda dt\Big]^{1/q}\label{mterm}\\
&\relphantom{=} -\Big[ \int_0^1\int_{\Omega}\sup_{R\in\mathcal{D}_{\mu}}\Big(\frac{1}{|R|}\int_{R}\sum_{k=M}^{L}{2^{st_kq}\big| (\phi_{t_k}+\phi_{t_k+1})\ast d_{t_k}(x)\big|^q}dx\Big)     d\lambda dt\Big]^{1/q}\label{eterm}
\end{align} for some $C>0$.

Clearly, $(\ref{eterm})\lesssim 1$ uniformly in $L$ by the idea in (\ref{minorkey}) and applying the argument that led to the bound in (\ref{upperbd}).

Now by applying (\ref{doublecondition}) with $g_{t_n}=h_{n}^{\omega,t,\mu}$ and the method in (\ref{dfbound}) with Lemma \ref{maximal2} one obtains that (\ref{mterm}) is greater than a constant times
\begin{align*}
& \Big[ \int_0^1\int_{\Omega}\sup_{R\in\mathcal{D}_{\mu}}{\Big(\frac{1}{|R|}\int_R\sum_{k=M}^{L}{\Big| \sum_{n=1}^{L-k}2^{-st_n}\Gamma_{t_k}\ast h_{{n+k}}^{\omega,t,\mu}(x)\Big|^q}      dx\Big)}d\lambda dt\Big]^{1/q}\\
&\geq \Big[2^{\mu d} \int_{[0,2^{-\mu}]^d}\int_{\Omega}{\sum_{k=M}^{L}{\Big(\int_{0}^1\Big| \sum_{n=1}^{L-k}\sum_{Q\in\mathcal{R}^{\mu}(n+k)}r_Q(t)\theta_Q(\omega)2^{-st_n}\Gamma_{t_k}\ast \chi_Q(x)\Big|^q dt\Big)}      }d\lambda dx\Big]^{1/q}\\
&\approx \Big[2^{\mu d} \sum_{k=M}^{L}\int_{[0,2^{-\mu}]^d}\int_{\Omega}{ {\Big(\sum_{n=1}^{L-k}\sum_{Q\in\mathcal{R}^{\mu}(n+k)}\theta_Q(\omega)2^{-2st_n}\big|\Gamma_{t_k}\ast \chi_Q(x)\big|^2  \Big)^{q/2}}    }d\lambda dx\Big]^{1/q}.
\end{align*}
For each $P\in\mathcal{R}^{\mu}(k)$ let $\mathcal{V}^{\mu}_n(k,P):=\{Q\in\mathcal{R}^{\mu}(k+n):Q\subset P\}$. 
Then we see that the last expression is 
\begin{align*}
&\geq \Big[2^{\mu d} \sum_{k=M}^{L}\sum_{P\in\mathcal{R}^{\mu}}\int_{P}\int_{\Omega}{ {\Big(\sum_{n=1}^{L-k}\sum_{Q\in\mathcal{V}^{\mu}_n(k,P)}\theta_Q(\omega)2^{-2st_n}\big|\Gamma_{t_k}\ast \chi_Q(x)\big|^2  \Big)^{q/2}}    }d\lambda dx\Big]^{1/q}\\
&\gtrsim \Big[2^{\mu d}\sum_{k=M}^{L}2^{-t_kd}\sum_{P\in\mathcal{R}^{\mu}(k)}\int_{\Omega}\Big( \sum_{n=1}^{L-k}{2^{-2st_n}2^{-2t_nd}\sum_{Q\in\mathcal{V}_n(k,P)}{\theta_Q(\omega)}}\Big)^{q/2}d\lambda \Big]^{1/q}\\
&\gtrsim \Big[2^{\mu d}\sum_{k=M}^{L}2^{-t_kd} \Big( \sum_{n=1}^{L-k}2^{-2st_n}2^{-2t_nd}\Big[\sum_{P\in\mathcal{R}^{\mu}(k)}\int_{\Omega} \Big( {\sum_{Q\in\mathcal{V}_n(k,P)}{\theta_Q(\omega)}}\Big)^{q/2}d\lambda \Big]^{2/q} \Big)^{q/2}\Big]^{1/q}\\
&\gtrsim \Big[ 2^{\mu d}\sum_{k=M}^{L}{2^{-t_kd}\Big(\sum_{n=1}^{L-k}{2^{-2st_n}2^{-2t_nd}2^{2t_nd/q}\Big[ \sum_{P\in\mathcal{R}^{\mu}(k)}{1/L \big(1-1/L \big)^{2^{t_nd}}}\Big]^{2/q}} \Big)^{q/2}     }\Big]^{1/q}\\
&\approx \Big[\frac{1}{L}\sum_{k=M}^{L}{\Big( \sum_{n=1}^{L-k}2^{-2(s+d-d/q)t_n}\Big( 1-\frac{1}{L}\Big)^{(2/q)2^{t_nd}}\Big)^{q/2}} \Big]^{1/q}\\
&\gtrsim L^{-(s+d-d/q)/2d}(\log{L})^{1/2}
\end{align*} by following the process to get (\ref{toshow}). This proves (\ref{lowerbd}).

\subsection{Sharpness of $s>d/p-d$ in $F$-space and $B$-space when $0<p\leq 1$}

We will prove Theorem \ref{sharpresult1} (1) for $0<p\leq 1$ and $p\leq q\leq \infty$ and Theorem \ref{sharpresult1} (2) for $0<p\leq 1$ and $0<q\leq \infty$.
\subsubsection{{Proof of Theorem \ref{sharpresult1} (1)}}
Suppose $0<p\leq 1$, $p\leq q\leq \infty$, and $s\leq d/p-d$.
We will apply (\ref{upper}) and (\ref{lower}).
Pick a nonnegative smooth function $g$ supported in a ball of radius $2^{-M}$, centered at the origin, and define 
\begin{equation}\label{ffgg}
g_k(x):=2^{k{d}/{p}}g(2^kx).
\end{equation}
Then for all ${\sigma}>0$ 
\begin{align}\label{maxin}
\mathfrak{M}_{{\sigma},2^{t_n}}g_{{t_n}}(x)&\leq  2^{{dt_n}/{p}}\frac{1}{(1+2^{{t_n}}|x|)^{{\sigma}}}\sup_{|y|\leq 2^{-M}}{(1+|y|)^{{\sigma}}|g(y)|}\nonumber\\
     &\lesssim 2^{{dt_n}/{p}}\frac{1}{(1+2^{{t_n}}|x|)^{{\sigma}}}.
\end{align} 
We choose $\sigma>d/p$.
By (\ref{upper}) and $l^p\hookrightarrow l^q$ one has 
\begin{align*}
\big\Vert f_L\big\Vert_{F_p^{{s},q}}&\lesssim \Big[ \int_{\mathbb{R}^d}{ \Big(  \sum_{n=M}^{L}{2^{{d}{t_n}q/p}\frac{1}{(1+2^{{t_n}}|x|)^{{\sigma} q}}}   \Big)^{{p}/{q}}  }dx   \Big]^{{1}/{p}}\\
&\leq \Big( \int_{\mathbb{R}^d}{   \sum_{n=M}^{L}{2^{d{t_n}}\frac{1}{(1+2^{{t_n}}|x|)^{{\sigma} p}}}    }dx   \Big)^{{1}/{p}}\lesssim L^{{1}/{p}}.
\end{align*}

For the lower bound we see that
 \begin{align}\label{mainlower}
&\Big\Vert   \Big(  \sum_{l=M}^{L-M}{\Big| \sum_{n=1}^{L-l}{2^{-{s} {t_n}}{\Gamma}_{{t_l}}\ast g_{t_{n+l}}}   \Big|^q}    \Big)^{{1}/{q}}\Big\Vert_{L^p}\nonumber\\
&\geq \Big(\sum_{j=M}^{L-M}{\Big\Vert     \Big(  \sum_{l=M}^{L-M}{\Big| \sum_{n=1}^{L-l}{2^{-{s} {t_n}}{\Gamma}_{{t_l}}\ast g_{t_{n+l}}}   \Big|^q}    \Big)^{{1}/{q}}   \Big\Vert_{L^p( 2^{-t_j-M-1}\leq|x|<2^{-t_j-M}   )}^{p}} \Big)^{1/p} \nonumber\\
&\geq \Big( \sum_{j=M}^{L-M}{\int_{2^{-{t_j}-M-1}\leq |x|\leq 2^{-{t_j}-M}}{     \Big| \sum_{n=1}^{L-j}{2^{-{s} {t_n}}{\Gamma}_{{t_j}}\ast g_{t_{n+j}}(x)}   \Big|^p }dx } \Big)^{{1}/{p}}.
\end{align}
If $|y|\leq 2^{-M}$ and $|x|\leq 2^{-{t_j}-M}$ then $|2^{{t_j}}x-y|\leq 2^{-M+1}$ and for such $x$ and $y$ one has ${\Gamma}(2^{{t_j}}x-y)\geq 1$.
Therefore
 \begin{align}\label{low}
{\Gamma}_{{t_j}}\ast g_{t_{n+j}}(x)&= 2^{{d}{t_n}/p}2^{{d}{t_j}/p}\int_{|y|\leq 2^{-M}}{{\Gamma}(2^{{t_j}}x-y)g(2^{{t_n}}y)}dy\nonumber\\
&\geq  2^{{d}t_n/{p}}2^{{d}t_j/{p}}2^{-{t_n}d}\Vert g \Vert_{L^1}
\end{align}
and this gives
\begin{align*}
(\ref{mainlower})&\gtrsim_M \Big[ \sum_{j=M}^{L-M}{\Big(\sum_{n=1}^{L-j}{2^{-{t_n}({s}+d-{d}/{p})}}\Big)^p}  \Big]^{{1}/{p}}\\
  &\geq  \Big[ \sum_{j=\lfloor L/3\rfloor}^{\lfloor L/2\rfloor}{\Big(\sum_{n=1}^{\lfloor L/3\rfloor}{2^{-{t_n}({s}+d-{d}/{p})}}\Big)^p}  \Big]^{{1}/{p}}  \gtrsim L^{{1}/{p}}L^{-({s}+d-{d}/{p})}\log{L}
\end{align*} for sufficiently large $L$.

By (\ref{lower}), \begin{equation*}
\big\Vert  T_{[a]}f_L \big\Vert_{F_p^{{s},q}} \gtrsim  L^{{1}/{p}}L^{-({s}+d-{d}/{p})}\log{L}
\end{equation*} and
 we are done by letting $L\to \infty$ since ${s}+d-{d}/{p}\leq 0$.

\subsubsection{{Proof of Theorem \ref{sharpresult1} (2)}}
Suppose $0<p\leq 1$ and $s\leq d/p-d$.
Define $a$ and $f_L$ to be as (\ref{gggfff}) and (\ref{fffggg}) with $g$ as in (\ref{ffgg}). Corresponding to (\ref{upper}) we have the analogous estimate 
\begin{equation*}\label{bin}
\big\Vert   f_L\big\Vert_{B_p^{{s},q}} \lesssim_{\sigma} \Big( \sum_{n=M}^{L}{\big\Vert \mathfrak{M}_{{\sigma},2^{t_n}}{g_{{t_n}}}\big\Vert_{L^p}^q}  \Big)^{{1}/{q}}
\end{equation*} for ${\sigma}>{d}/{p}$ and then (\ref{maxin}) yields
\begin{equation*}
\Vert f_L\Vert_{B_p^{s,q}} \lesssim \Big( \sum_{n=M}^{L}{2^{{d}q{t_n}/p}\Big\Vert  \frac{1}{(1+2^{{t_n}}|\cdot|)^{{\sigma}}} \Big\Vert_{L^p}^{q}}\Big)^{{1}/{q}}\lesssim L^{{1}/{q}}.
\end{equation*} 

Furthermore, similar to (\ref{lower}) one has 
\begin{align*}
\big\Vert {T_{[a]}}f_L \big\Vert_{B_p^{{s},q}} &\geq A\Big( \sum_{l=M}^{L-M}{2^{st_lq}\Big\Vert \sum_{n=l+1}^{L}{\big(\phi_{t_l}+\phi_{t_l+1}\big)\ast d_{t_n}}\Big\Vert_{L^p}^{q}}\Big)^{1/q}\\
&\relphantom{=} - B \Big( \sum_{l=M}^{L-M}{2^{st_lq}\big\Vert {\big(\phi_{t_l}+\phi_{t_l+1}\big)\ast d_{t_l}}\big\Vert_{L^p}^{q}}\Big)^{1/q}\\
&=: \mathcal{I}-\mathcal{J}
\end{align*} 
for some $A,B>0$,
and using (\ref{minorkey}) 
\begin{equation*}
\mathcal{J}\lesssim\Big(\sum_{l=M}^{L-M}{\big\Vert \mathfrak{M}_{\sigma,2^{t_l}}g_{t_l}\big\Vert_{L^p}^q} \Big)^{1/q}\lesssim L^{1/q}.
\end{equation*}

To get the lower bound of $\mathcal{I}$ we apply (\ref{doublecondition}) and the idea in (\ref{dfbound}), and then it follows
\begin{align*}
\mathcal{I}&\gtrsim\Big(\sum_{l=M}^{L-M}{2^{st_lq}\Big\Vert \sum_{n=l+1}^{L}{2^{-st_n}\big(\phi_{t_l}+\phi_{t_l+1}\big)\ast g_{t_n}}\Big\Vert_{L^p}^q} \Big)^{1/q}\\
&\gtrsim \Big( \sum_{l=M}^{L-M}{\Big\Vert \sum_{n=1}^{L-l}{2^{-{s} {t_n}}\Gamma_{t_l}\ast g_{t_{n+l}}}   \Big\Vert_{L^p}^{q}} \Big)^{{1}/{q}}\nonumber\\
  &\geq \Big( \sum_{l=M}^{L-M}{\Big\Vert \sum_{n=1}^{L-l}{2^{-{s} {t_n}}\Gamma_{t_l}\ast g_{t_{n+l}}}   \Big\Vert_{L^p(2^{-t_l-M-1}\leq |x|<2^{-t_l-M})}^{q}} \Big)^{{1}/{q}}.
  \end{align*}
Now (\ref{low}) yields that the last expression is greater than a constant times \begin{equation*}
\Big(\sum_{l=M}^{L-M}{\Big(\sum_{n=1}^{L-l}{  2^{-{t_n}({s}+d-{d}/{p})}}   \Big)^q}\Big)^{{1}/{q}} \gtrsim L^{{1}/{q}}L^{-({s}+d-{d}/{p})}\log{L}
\end{equation*} for ${s}+d-{d}/{p}\leq 0$. The proof ends by letting $L\to\infty$.

\subsection{Sharpness of $s>0$ in $B$-space when $0<q\leq 1<p\leq \infty$}

In this section we will prove Theorem \ref{sharpresult1} (2) for $0<q\leq 1<p\leq \infty$.
\subsubsection*{\textbf{The case $0<q\leq 1<p<\infty$}}\label{ee2}
Suppose $s\leq 0$.
Let $\{b_n\}$ be a sequence of real numbers, and let
\begin{equation*}
g_L(x):=\sum_{n=M}^{L}{b_n2^{-{t_n}d(1-{1}/{p})}\phi_{{t_n}} }(x),
\end{equation*} and 
\begin{equation*}
f_L(x):=2^{-{s} {t_L}}\Lambda_{{t_L}}\ast \big(g_L e^{-2\pi i\langle \mathbf{v}_{{t_L}},\cdot\rangle }\big)(x).
\end{equation*}
Then we observe that by Young's inequality
\begin{equation*}
\big\Vert  f_L \big\Vert_{B_p^{{s},q}} \lesssim \big\Vert \phi_{{t_L}}\ast\Lambda_{{t_L}}\ast \big(g_L e^{-2\pi i\langle \mathbf{v}_{t_{L}},\cdot \rangle}\big)  \big\Vert_{L^p}+\big\Vert \phi_{{t_L}-1}\ast\Lambda_{{t_L}}\ast \big(g_L e^{-2\pi i\langle \mathbf{v}_{{t_L}},\cdot \rangle}\big)  \big\Vert_{L^p}\lesssim\big\Vert  g_L \big\Vert_{L^p}.
\end{equation*} 
When $1<p<2$, by Littlewood-Paley theory and $l^p\hookrightarrow l^2$
\begin{align*}
\big\Vert g_L\big\Vert_{L^p}       &\lesssim \Big\Vert \Big(\sum_{n=M}^{L}{b_n^22^{-2t_nd(1-1/p)}|\phi_{t_n}|^2}\Big)^{1/2}\Big\Vert_{L^p}\\
&\leq \Big( \sum_{n=M}^{L}{b_n^p 2^{-{t_n}d(p-1)}\big\Vert  \phi_{{t_n}}\big\Vert_{L^p}^p}   \Big)^{{1}/{p}}\approx\Big(\sum_{n=M}^{L}{b_n^p}\Big)^{{1}/{p}}.
\end{align*}
(Here we may also use orthogonality for $p=2$ and triangle inequality for $p=1$, and then apply interpolation.)
When $2\leq p<\infty$, by Hausdorff-Young's inequality,
\begin{equation*}
\big\Vert  g_L \big\Vert_{L^p} \lesssim \Big\Vert  \sum_{n=M}^{L}{b_n2^{-{t_n}d(1-{1}/{p})}  \widehat{\phi_{{t_n}}} } \Big\Vert_{L^{p'}}=\Big(\sum_{n=M}^{L}{b_n^{p'}2^{-t_nd}\big\Vert \widehat{\phi_{t_n}}\big\Vert_{L^{p'}}^{p'}} \Big)^{1/p'}\approx\Big(\sum_{n=M}^{L}{b_n^{p'}}\Big)^{{1}/{p'}}
\end{equation*} where ${1}/{p}+{1}/{p'}=1$.
Thus for $1<p<\infty$ we have obtained \begin{equation}\label{finalupper}
\big\Vert f_L \big\Vert_{B_p^{s,q}}\lesssim \Big(\sum_{n=M}^{L}{b_n^{\widetilde{p}}}\Big)^{{1}/{\widetilde{p}}}\end{equation}
where $\widetilde{p}=\min{( p, \frac{p}{p-1} )}>1$.

On the other hand, let $a\in\mathcal{S}_{1,1}^{0}$ be defined as (\ref{gggfff}).
Then 
\begin{equation*}
{T_{[a]}}f_L(x) = 2^{-{s} {t_L}}\Lambda_{{t_L}}\ast\Lambda_{{t_L}}\ast\big( g_{L} e^{-2\pi i\langle \mathbf{v}_{{t_L}},\cdot \rangle}\big)(x)e^{2\pi i\langle \mathbf{v}_{{t_L}},x \rangle}=2^{-st_L}\sum_{k=0}^{L}{m_k(x)}
\end{equation*}
where
\begin{equation*}
m_0(x):=\Lambda_0\ast\Lambda_0\ast \big(g_L e^{-2\pi i\langle \mathbf{v}_0,\cdot \rangle}\big)(x)e^{2\pi i\langle \mathbf{v}_0,x\rangle},
\end{equation*}
\begin{align*}
m_k(x)&:= \Lambda_{{t_k}}\ast\Lambda_{{t_k}}\ast \big(g_L e^{-2\pi i\langle \mathbf{v}_{{t_k}},\cdot \rangle}\big)(x)e^{2\pi i\langle \mathbf{v}_{{t_k}},x \rangle}\\
   &\relphantom{=} -\Lambda_{t_{k-1}}\ast\Lambda_{t_{k-1}}\ast \big(g_L e^{-2\pi i\langle\mathbf{v}_{t_{k-1}},\cdot \rangle}\big)(x)e^{2\pi i\langle\mathbf{v}_{t_{k-1}},x \rangle}
\end{align*}   for $k\geq 1$.
Observe that for $k\geq 1$
\begin{equation*}\label{m}
\widehat{m_k}(\xi)      = \widehat{g}_L(\xi)\Big(  \big(\widehat{\Lambda}\big({\xi}/{2^{{t_k}}}-\mathbf{v}_0\big)\big)^2-\big(\widehat{\Lambda}\big({\xi}/{2^{t_{k-1}}}-\mathbf{v}_0\big) \big)^2 \Big):=\widehat{g}_L(\xi)\widehat{\Psi_{{t_k}}}(\xi)
\end{equation*} 
and 
\begin{equation*}
Supp{(\widehat{m_k})}\subset \{2^{-13}2^{{t_k}}\leq |\xi|\leq({27}/{8})2^{{t_k}}\}.
\end{equation*} Indeed,
if $|\xi|<2^{-13}2^{{t_k}}$ then $5/4<\big| {\xi}/{2^{{t_k}}}-\mathbf{v}_0   \big|<7/4$ and ${5}/{4}<\big| {\xi}/{2^{t_{k-1}}}-\mathbf{v}_0   \big|<{7}/{4}$ for which $\widehat{\Psi_k}(\xi)=0$. 
Moreover, if $|\xi|>(27/8) 2^{t_k}$ then $15/8<\big| {\xi}/{2^{{t_k}}}-\mathbf{v}_0   \big|$ and $15/8<\big| {\xi}/{2^{t_{k-1}}}-\mathbf{v}_0   \big|$ which imply $\widehat{\Psi_k}(\xi)=0$.
Therefore for ${s}\leq 0$
 \begin{align}\label{finallower}
\big\Vert {T_{[a]}}f_L\big\Vert_{B_p^{{s},q}} &= 2^{-{s} {t_L}}\Big(\sum_{l=0}^{\infty}{2^{{s} lq}\Big\Vert  \sum_{k=0}^{L}{\Pi_l m_k} \Big\Vert_{L^p}^q}\Big)^{{1}/{q}}\nonumber\\
&\geq  2^{-st_L}\Big(\sum_{l=M}^{L-1}{2^{st_lq}\Big\Vert \sum_{k=0}^{L}{\phi_{t_l}\ast m_k}\Big\Vert_{L^p}^q} \Big)^{1/q}\nonumber\\
&\geq  \Big(  \sum_{l=M}^{L-1}{\big\Vert \phi_{{t_l}}\ast (m_l+m_{l+1})  \big\Vert_{L^p}^q} \Big)^{{1}/{q}}\nonumber\\
   &= \Big(  \sum_{l=M}^{L-1}{b_l^q2^{-{t_l}d(1-{1}/{p})q}\big\Vert \phi_{{t_l}}\ast \phi_{{t_l}} \ast (\Psi_{{t_l}}+\Psi_{t_{l+1}}) \big\Vert_{L^p}^q} \Big)^{{1}/{q}}\nonumber\\
   &\approx  \Big(  \sum_{l=M}^{L-1}{  b_l^q   } \Big)^{{1}/{q}}.
\end{align}
We are done by choosing a sequence $\{b_n\}$ so that $ (\ref{finalupper}) \lesssim 1$ uniformly in $L$, but (\ref{finallower}) diverges as $L\to \infty$.

\subsubsection*{ \textbf{The case $0<q\leq 1<p=\infty$}{ ( actually, $0<q<\infty$ and $p=\infty$ )}}
Suppose $s\leq 0$.
Define a symbol $a\in\mathcal{S}_{1,1}^{0}$ to be 
\begin{equation*}
a(x,\xi):=\sum_{k=10}^{\infty}{\widehat{\phi_{t_k}^*}(\xi)e^{2\pi i\langle2^{t_k}e_1,x\rangle}}
\end{equation*} where $\phi_{t_k}^*:=\phi_{t_k-1}+\phi_{t_k}+\phi_{t_k+1}$ as before,
and for sufficiently large $L>0$ let 
\begin{equation*}
g_L(x):=2^{-st_L}e^{-2\pi i\langle 2^{t_L}e_1,x\rangle}\sum_{k=10}^{L-10}\phi(2^{t_k}(x-t_ke_1)).
\end{equation*}
We observe that $Supp(\widehat{g_L})\subset \{\xi:2^{t_L}(1-1/2^{99})\leq |\xi|\leq 2^{t_L}(1+1/2^{99})\}$ and thus
\begin{equation*}
\Vert g_L \Vert_{B_{\infty}^{s,q}}\lesssim \Big\Vert \sum_{k=10}^{L-10}{\phi\big(2^{t_k}(x-t_ke_1)\big)}\Big\Vert_{L^{\infty}}\lesssim 1.
\end{equation*}

On the other hand,
we see that 
\begin{equation*}
T_{[a]}g_L(x)=2^{-st_L}\sum_{k=10}^{L-10}\phi(2^{t_k}(x-t_ke_1))
\end{equation*} and thus
\begin{align*}
\Vert T_{[a]}g_L\Vert_{B_{\infty}^{s,q}}&\approx \Big(\sum_{k=10}^{L-10}{2^{st_kq}\big\Vert \phi_{t_k}^*\ast \big(T_{[a]}g_L\big)\big\Vert_{L^{\infty}}^q} \Big)^{1/q} \\
&\approx 2^{-st_L}\Big(\sum_{k=10}^{L-10}{2^{st_kq}\big\Vert \phi\big(2^{t_k}(\cdot-t_k e_1) \big)\big\Vert_{L^{\infty}}^q} \Big)^{1/q}\approx 
\begin{cases} 
L^{1/q} \quad &~s=0\\
2^{-st_L} \quad &~s<0
\end{cases}.
\end{align*}
This completes the proof.

\subsection*{Acknowledgements}
The author was supported in part by NSF grant DMS 1500162.
The author would like to thank his graduate advisor Andreas Seeger for the guidance and helpful discussions and Professor Jon Johnsen for very informative discussions and preprints of his works.
The author also would like to express gratitude to the anonymous referees for the careful reading and suggestions to improve the presentation of the work.

\end{document}